\documentclass[12pt,a4paper]{amsart}

\usepackage{amsfonts}
\usepackage{amssymb}
\usepackage{amsthm}
\usepackage{amsmath}
\usepackage{mathrsfs}
\usepackage{bbold}
\usepackage{enumerate}
\usepackage[colorlinks=true]{hyperref}

\usepackage{graphicx}

\usepackage{graphicx}
\usepackage[all,cmtip]{xy}

\def\IC{\mathbb{C}}

\def\IH{\mathbb{H}}

\def\BH{\mathcal{B}(H)}

\newcommand{\SPAN}{\mathrm{span}}

\newcommand{\id}{\mathop{\mathrm{id}}}
\newcommand{\TRO}{\mbox{TRO}}

\newcommand{\TROu}{{T^*}}

\newcommand{\Fock}{{\mathscr{F}_{-}}}
\newcommand{\CAR}{{\mathscr{A}}}

\newcommand{\tp}[3]{\{#1 , #2 , #3\}}
\newcommand{\trop}[3]{[#1 , #2 , #3]}
\newcommand{\utro}[1]{\alpha_{#1}}
\newcommand{\utr}{\alpha}

\newcommand{\tmpref}[1]{}

\newcommand{\JCstar}{$JC^*$}
\newcommand{\JC}{$JC$}
\newcommand{\JBstar}{$JB^*$}
\newcommand{\JWstar}{$JW^*$}

\newcommand{\Sp}{\mathrm{Sp}}
\newcommand{\sgn}{\mathrm{sgn}}

\numberwithin{equation}{section}
\newtheorem{proposition}{Proposition}[section]
\newtheorem{lemma}[proposition]{Lemma}
\newtheorem{theorem}[proposition]{Theorem}
\newtheorem{corollary}[proposition]{Corollary}
\theoremstyle{definition}
\newtheorem{remark}[proposition]{Remark}
\newtheorem{remarks}[proposition]{Remarks}

\newtheorem{definition}[proposition]{Definition}

\begin{document}

\title{Operator space structure of \JCstar-triples and TROs, I}
\author[L.~J.~Bunce]{Leslie J. Bunce}
\address{Department of Mathematics, University of Reading, Reading RG6
2AX}
\email{l.j.bunce@reading.ac.uk}
\author[B. Feely]{Brian Feely}
\address{School of Mathematics\\Trinity College\\Dublin 2}
\email{brian.feely@gmail.com}
\author[R.~M.~Timoney]{Richard M. Timoney}
\address{School of Mathematics\\Trinity College\\Dublin 2}
\email{richardt@maths.tcd.ie}
\thanks{The work of the second and third authors was
supported
by the Science Foundation Ireland under grant 05/RFP/MAT0033.}
\date{Final draft \today}

\subjclass[2000]{47L25; 17C65; 46L70}

\keywords{
CAR algebra;
Cartan factor;
JC-operator space;
operator space ideal;
reversibility;
ternary ring of operators;
universal TRO}

\begin{abstract}
We embark upon a systematic investigation of operator space structure
of \JCstar-triples via a study of the TROs (ternary rings of operators)
they generate. Our approach is to introduce and develop a variety of
universal objects, including universal TROs, by which means we are
able to describe all possible operator space structures of a
\JCstar-triple. Via the concept of reversibility we obtain
characterisations of universal TROs over a wide range of examples. We
apply our results to obtain explicit descriptions of
operator space structures of Cartan factors
regardless of dimension.
\end{abstract}
\maketitle

\section{Introduction}

TROs (\emph{ternary rings of operators}), the Banach subspaces of
$C^*$-algebras closed under the ternary product  $\trop{a}{ b}{ c} =
ab^*c$, are natural and significant objects in the category of
operator spaces. We mention (see \cite{Hamana99,RuanTAMS1981},
\cite[Chapter 6]{ER}) that the algebraic isomorphisms between TROs
coincide with the surjective complete isometries, the isomorphisms
between operator spaces, and that each operator space $E$ may be
realised (completely isometrically) as an operator subspace of its
injective envelope $I(E)$ which, in turn, may be realised as a TRO: in
which case if $F$ is an operator subspace of a $C^*$-algebra $A$ then
every complete isometry from $F$ onto $E$ is the restriction of a
TRO homomorphism from the TRO generated by $F$ onto the triple
envelope $\mathcal{T}(E)$ of $E$, the TRO generated by $E$ in its
realisation in $I(E)$.

A complex Banach space $E$ is said to be a \emph{\JCstar -triple}
if there is
a surjective isometry $\phi \colon F \to E$ where $F$ is a subspace
of a TRO closed under the (Jordan) triple product $\tp{a}{ b}{ c}
=(1/2)( \trop  a b c + \trop c b a)$. In these circumstances $F$ is
called a \emph{concrete} \JCstar -triple or, more specifically, a
\JCstar -subtriple of the TRO in question. By a result originating
from the general theory of \JBstar -triples, the triple product
induced on $E$ defined by $\tp{\phi(a)} {\phi(b)} {\phi(c)} = \phi(
\tp a b c)$ is independent of $\phi$ \cite{Kaup}. In particular, all
linear isometric images of TROs are \JCstar -triples. An abstract
Hilbert space $H$ is a \JCstar -triple with triple product given by
 $\tp {\xi}{ \eta}{ \zeta} = ( \langle \xi, \eta
 \rangle \zeta + \langle \zeta, \eta \rangle \xi)/2$
as may be seen from either of the surjective isometries
$H \to \BH e$ ($\xi \mapsto \xi \otimes h$),
$H \to e\BH $ ($\xi \mapsto h \otimes \bar{\xi}$) where $h \in H$ is a unit
vector and $e$ is the projection $h \otimes h$. Although linearly
isometric, $\BH e$ and $e \BH$ are not completely isometric if $H$ has
dimension greater than unity \cite[\S10]{Pisier}.

A contractively complemented subspace $E$ of $\BH$ is an example of a
\JCstar -triple
(not necessarily a \JCstar -subtriple of
$\BH$)
\cite[Corollary 2.4]{StachoPP1982},
\cite{Kaup1984MathScand},
\cite{FRJFA1985}.
 Operator space structure of \JCstar -triples of this kind has
been investigated in the important series of articles
\cite{NealRussoTAMS03,NealRicardRusso06,NealRussoPAMS06} in the case
when $E$ is reflexive, equivalently, when $E$ is a finite
$\ell_\infty$-sum of reflexive Cartan factors. Cartan factors also
figure prominently in the study of completely contractively
complemented subspaces of Schatten spaces \cite{LeMerdyRicardRoydor}.

Our aim in this paper is to explore operator space structures of an
arbitrary \JCstar -triple. One of our aims is to place the ground-breaking
results of \cite{NealRussoTAMS03,NealRicardRusso06,NealRussoPAMS06}
into the general setting of a full theory. To this end we shall
initiate the study of certain universal objects associated with a
\JCstar -triple and, in particular, that of the universal TRO of a
\JCstar -triple. We shall show that this new device (which differs
from the above mentioned triple envelope --- see
Theorem~\ref{TripleEnvThm})
enables a systematic examination of the TROs generated by \JCstar
-triples and facilitates a general enquiry into their operator space
structures.
As a technical aid of possibly independent interest we shall also
introduce the notion of a reversible \JCstar -triple and shall apply
our general theory to obtain concrete identifications of the universal
TROs of all Cartan factors (those that arise as \JCstar -triples ---
the two exceptional Cartan factors of \JBstar-triple theory are not
\JCstar -triples) and we shall exhibit the possible operator space
structures of Cartan factors of rank greater than unity.

We shall tend to use \cite{ER,Pisier} as our standard references for
the theory of operator spaces (see also
\cite{Paulsen,BlecherleMerdy}). Briefly, we recall that an \emph{operator
space} is a complex Banach space $E$ together with a linear isometric
embedding of $E$ into some $\BH$ and that the corresponding operator
space structure on $E$ is determined by the matrix norms on $M_n(E)$
inherited from $M_n(\BH)$. A linear map between operator spaces, $\pi
\colon E \to F$, is said to be \emph{completely bounded} with completely
bounded norm given by $\|\pi\|_{cb} = \sup\{ \|\pi_n\|: n \geq 1 \}$ if
the latter is finite, where $\pi_n$ is the tensored map $\pi \otimes
I_n \colon M_n(E) \to M_n(F)$, for each $n$.
If $\|\pi\|_{cb} \leq 1$, then $\pi$ is said to be a \emph{complete
contraction} and to be a \emph{complete isometry} if every $\pi_n$  is
isometric.

As already noted TROs occupy a special place in operator space theory.
To expand upon the triple product of an abstract \JCstar-triple we
remark that the the results of \cite{Kaup} show that $\tp x y z$ on
$E$, as was defined above, is the unique product on $E$ that is
symmetric and bilinear in $x$ and $z$ and conjugate linear in $y$ for
which the operator $D(x, y)$ on $E$ given by $D(x,y)(z) = \tp x y z$
satisfies
\[
[ D(x,y), D(a, b) ] = D(\tp x y a, b) - D(a, \tp b x y),
\]
and such that $D(x, x)$ is a positive Hermitian operator in
$\mathcal{B}(E)$ with norm  $\|x\|^2$.
We further remark that the use of abstract \JCstar-triples in this
paper is a matter of convenience and that nothing essential would be
lost by exclusive concentration upon their concrete realisations. The
fundamental difference between \JCstar-triples and TROs is that the
underlying structure of the former is determined by their linear
isometric class while that of the latter is fixed up to complete
isometry.

We have gathered background material in the next section and start on
the new concepts, including universal TROs in
\S\ref{SectionUniversal}. We introduce the notion of universal
reversibility in \S\ref{SectionReversibility} as an aid to the
identification of the universal TROs and, in \S\ref{SectionCartanFs},
we bring our methods to bear upon Cartan factors.
One outcome of independent interest is a
classification of universally reversible Cartan factors
(Theorem~\ref{CharactRevCFs}).
In the final section
we show that the operator space structures of a \JCstar-triple arising
from its concrete realisations are completely determined by the
\emph{operator space ideals} (as we call them) of its universal TRO
and we exhibit links with injective and triple envelopes.
We apply this to elucidate a full description of the operator
space strucutures of all non Hilbertian Cartan factors, answering a
question raised in \cite{LeMerdyRicardRoydorv2}.

The Hilbertian case requires special consideration and we postpone
a detailed analysis to a forthcoming paper. In addition, we plan other
articles on developments of our ideas excluded from the present paper
such as exactness (as a functor) of the universal TRO, its behaviour
with respect to relevant tensor products and the introduction and use
of the weak*-TRO of a \JWstar-triple (a \JCstar-triple with a predual)
which shall also enable us to classify all universally reversible
\JCstar-triples and to analyse operator space structures of a wide
range of \JCstar-triples and \JWstar-triples.

The methods pioneered in \cite{NealRussoTAMS03},
and the techniques of \cite{LeMerdyRicardRoydorv2}, do not
readily extend from the cases of the finite rank Cartan factors for
which they were devised. In contrast, our univeral approach provides a
coherent general framework for treating any \JCstar-triple. In particular,
for Cartan factors we do not require triple embeddings to be
weak*-continuous nor the existence of norm dense grids of minimal
tripotents.

\section{Background}
\label{secBackground}

In this section, 
we have assembled preliminary material for later use.
General references for \JCstar-triples are the surveys
\cite{RussoSurvey,RodSurvey1994} and the papers
\cite{FRgn,Harris74,Harris81,Kaup} whilst
\cite{Zettl,BlecherleMerdy,ER, Paulsen} are helpful sources for TROs.
Another class of present relevance is that of \JCstar-algebras, the
norm closed subspaces of $C^*$-algebras which are closed under the
involution and Jordan product, $a \circ b = (1/2)(ab +ba)$.
\JCstar-algebras are the complexifications of $JC$-algebras, about which
\cite{HOS} is an extensive monograph.
Since the triple product on $\BH$ can be expressed
in terms of the Jordan product and $*$-operation, via
$\tp{x}{ y}{ z} = (x
\circ y^*) \circ z + (z \circ y^*) \circ x - (x \circ z) \circ y^*,
$
it follows that \JCstar -algebras are \JCstar -triples.

By an \emph{ideal} we will always mean a \emph{norm closed ideal}.
Thus an ideal of a \JCstar-triple or a TRO $E$  is a norm closed
subspace $I$ of $E$ for which, respectively, $\tp a b c$ or $\trop a b
c$ lies in $I$ whenever $a$, $b$ or $c$ belongs to $I$ (and $a,b,c \in
E$).
An ideal of a \JCstar-algebra $A$ is a norm closed subspace $I$ such
that $A \circ I \subset I$. The following removes any ambiguity.

\begin{lemma}[{\cite[Proposition 5.8]{Harris81}}]
The ideals of a TRO, \JCstar-algebra or $C^*$-algebra $E$ coincide
with the ideals of $E$ when it is regarded as a \JCstar-triple.
\end{lemma}

The quotient of a \JCstar-triple by an ideal is again a \JCstar-triple
\cite[Theorem 2]{FRgn} and similarly for TROs \cite[Proposition
2.1]{Hamana99}.
Naturally, triple homomorphisms between \JCstar-triples and TRO
homomorphisms between TROs are, respectively, the linear maps
preserving the underlying triple and TRO products.
The image of a triple homomorphism (hence of a TRO homomorphism) is
norm closed.
The first part of the next statement was proved in \cite[Proposition
3.4]{Harris81}, \cite[Theorem 4]{Harris74} and the second part in
\cite[Proposition 
2.1]{Hamana99}.

\begin{lemma}
\label{TROmorphsAreCC}
\begin{enumerate}[(a)]
\item Between \JCstar-triples, triple homomorphisms are contractive,
injective triple homomorphisms are isometric and surjective linear
isometries are triple isomorphisms.
\item
Between TROs, TRO homomorphisms are completely contractive, 
injective TRO homomorphisms are completely isometric
and surjective complete 
isometries are TRO isomorphisms.
\end{enumerate}
\end{lemma}

By a TRO \emph{antihomomorphism},
$\phi \colon T \to S$, between TROs we mean a linear map such that
$\phi(\trop a b c) = \trop{\phi(c)}{\phi(b)}{\phi(a)}$ whenever $a, b,
c \in T$. Typically these arise as restrictions of *-antihomomorphisms
between containing $C^*$-algebras.

Recall \cite{Kaup} that the \JCstar-subtriple generated by an element
$x$ is a \JCstar-triple $E$ is linearly isometric to an abelian
$C^*$-algebra, giving rise to a functional calculus and, in
particular, to the existence of a `cube root' $y \in E$ such that
$x = \tp y y y$. Elements $a, b \in E$ are said to be orthogonal if
$\tp a a b = 0$, a condition equivalent to $a^*b = a b^* = 0$ when $E$
is realised as  a \JCstar-subtriple of a $C^*$-algebra \cite[p.
18]{Harris74}, implying $\|a + b\| = \max(\|a\|, \|b\|)$.
Thus ideals $I$ and $J$ of $E$ are orthogonal if and only if $I \cap J
= \{0\}$, in which case $I+J$ is an $\ell_\infty$ sum.

If $E$ is a subset of a $C^*$-algebra we write
$\TRO(E)$ for the TRO generated by $E$: it
the closed linear span of
elements of the form
$x_1 x_2^* x_3 \ldots x_{2n}^* x_{2n+1}$ % \quad
where the $x_i \in E$ and $n \geq 0$.
We note that:

\begin{lemma}
\label{MsummandsProp}
If $E$ and $F$ are orthogonal subsets of a $C^*$-algebra, then
$\TRO(E)$ and $\TRO(F)$ are orthogonal ideals of $\TRO(E+F) = \TRO(E)
+ \TRO(F)$, which is an $\ell_\infty$ sum
(where $E +F = \{ a + b : a \in E, b \in F\}$).
\end{lemma}

For a sub-TRO, $T$, of a $C^*$-algebra we denote
the $C^*$-subalgebra of $\BH$ generated by
$\{ab^*: a, b \in T\}$ and $\{a^*b: a, b \in T\}$
by $\mathscr{L}_T$ and $\mathscr{R}_T$ respectively.
By definition $\mathscr{L}_T T, T \mathscr{R}_T \subseteq T$
and so, taking `cube roots' $\mathscr{L}_T T = T = T \mathscr{R}_T$.

A tripotent in a
\JCstar -triple $E$ is an element $u$ such that $u  =\tp{u}{ u}{ u }$
(a partial isometry if $E$ is realised as a \JCstar-subtriple of a
$C^*$-algebra).
There is an elaborate Peirce decomposition theory associated with a
tripotent $u$ in a \JCstar-triple $E$. For our present purposes we
should recall that the Peirce 2-space $\tp u {\tp u E u} u$, denoted
by $E_2(u)$, is a \JCstar-algebra with identity $u$ and with
involution and Jordan product defined by
\[
a^{\#} = \tp u a u , \quad a \circ b = \tp a u b.
\]
Concretely, when $u$ is a partial isometry in a \JCstar-subtriple $E$
of a TRO $T$ in $\BH$, putting $e = uu^*$ and $f = u^*u$, we have
$T_2(u) = eTf$ is a $C^*$-algebra \cite[p. 120]{Zettl} with
involution and product given by
\[
a^{\#} = \trop u a u , \quad a \bullet b = \trop a u b
\]
containing $E_2(u) = eEf$ as a \JCstar-subalgebra, and we note that
\[
a \bullet b^{\#} \bullet c = a b^*c = \trop a b c
\qquad \mbox{for all } a, b, c \in T_2(u).
\]

The following is a straightforward consequence.

\begin{lemma}
\label{TROPeirce2}
Let $E$, $T$, $u$, $e$ and $f$ be as in the above paragraph.
\begin{enumerate}[(a)]
\item If $\pi \colon A \to T_2(u)$ is a *-homomorphism of
$C^*$-algebras, then $\pi \colon A \to T$ is a TRO homomorphism.

\item The maps $x \mapsto u^* x$ and $x \mapsto x u^*$ are
*-isomorphisms from the $C^*$-algebra $T_2(u)$ onto $C^*$-algebras of 
$f \mathscr{R}_T f$
and
$e \mathscr{L}_T e$, respectively.
\end{enumerate}
\end{lemma}

For a Banach space $E$, we use $E^{**}$ for its bidual, and we
consider $E \subseteq E^{**}$ canonically.

\begin{lemma}
\label{TROofAlgIsAlg}
Let $A$ be a \JCstar-subalgebra of $\BH$. Then $\TRO(A)$ is a
$C^*$-subalgebra.
\end{lemma}

\begin{proof}
Put $T = \TRO(A)$ and let $B$ be the $C^*$-subalgebra generated by $A$
in $\BH$. 
Then $A^{**} \subseteq T^{**} \subseteq B^{**}$ so that the TRO
$T^{**}$ contains the common identity of $A^{**}$ and $B^{**}$
($A^{**}$ has an identity and
generates $B^{**}$ as a $W^*$-algebra).
Hence $ T^{**}$ is a $C^*$-subalgebra of  $B^{**}$ implying that $T =
T^{**} \cap B$ is a $C^*$-subalgebra of  $B$.
\end{proof}

\begin{proposition}
\label{JCStAlgSubtriple}
Let $E$ be a \JCstar-subtriple of $\BH$ where $E$ is linearly
isometric to a \JCstar-algebra. Then there is a partial isometry $u
\in \BH$ such that $u^* \TRO(E)$ is a $C^*$-subalgebra of $\BH$ and
$\TRO(E) \to u^* \TRO(E)$ is a TRO isomorphism. If $E$ is weakly
closed, then $u$ can be chosen in $E$.
\end{proposition}

\begin{proof}
Let $\phi \colon A \to E$ be a surjective linear isometry, hence
a triple isomorphism, where
$A$ is a \JCstar-algebra and let $\psi \colon A^{**} \to \BH$ be the
weak*-continuous extension of $\pi$.
Then $\psi$ is a triple homomorphism and letting $u = \psi(1)$, where
1 is the identity of $A^{**}$, we have $uu^* x u^* u = x$ implying
$uu^* x = x u^*u$, for all $x \in E$, and hence for all $x \in
\TRO(E)$.
By direct calculation, $u^* E$ is a \JCstar-subalgebra of $\BH$ and
$\TRO(u^*E) = u^* \TRO(E)$, which is a $C^*$-subalgebra of $\BH$ in
view of Lemma~\ref{TROPeirce2}. The final statement follows from the
fact that the weak closure of $E$ contains $\psi(E^{**})$.
\end{proof}

If $e$ is a projection in a von Neumann algebra $W$, the weak*-closure
of $WeW$ is the weak*-closed ideal of $W$ generated by $e$ and equals
$W c(e)$ where $c(e)$ is the central cover of $e$.
Given projections $e$ and $f$ in $W$ we use $e \sim f$ to indicate von
Neumann equivalence, that is, that $e = u^*u$, $f = uu^*$ for some
(partial isometry) $u \in W$.

\begin{corollary}
The following are equivalent for a non-zero projection $e$ in a von
Neumann algebra $W$ with $c(e) = 1$
\begin{enumerate}[(a)]
\item $We$ is linearly isometric to a \JCstar-algebra;
\item $e \sim 1$;
\item $We$ is TRO isomorphic to $W$;
\item $We$ is TRO isomorphic to $eW$.
\end{enumerate}
\end{corollary}

\begin{proof}
(a) $\Rightarrow$ (b). Assume (a). By
Proposition~\ref{JCStAlgSubtriple} there is a partial isometry $u \in
We$ such that $u^*We = u^*uWe$
is a $C^*$-subalgebra of $W$,
giving
$u^*u= e$, and $\pi \colon We \to u^*We$ ($x \mapsto u^* x$)
is a TRO isomorphism. Since $\pi(xe) = \pi (uu^*xe)$ for $x \in W$ we
have $We = uu^* We$, thus $WeW  = u^*u WeW$ and so $u^*u = 1$ since
$c(e) = 1$. Therefore $e \sim 1$.

(d) $\Rightarrow$ (b). Let $\pi \colon We \to eW$ be a TRO isomorphism
and put $\pi(e) = u$. Given $x \in We$, $\pi(x) = \pi(x) u^* u$ so that
$WeW = WeW(u^*u)$ and thus $1 = u^*u$ (since $c(e) = 1$).
Since $u^* u \leq e$, we have $e \sim 1$.

The implications (b) $\Rightarrow$ (c) $\Rightarrow$ (a) and (b)
$\Rightarrow$ (d) are clear.
\end{proof}

Below, if $S$ is a subset of a $C^*$-algebra $A$, $S^{\#}$ denotes
$\{ x^* : x \in S \}$ and $\langle S \rangle$ denotes the (norm
closed) ideal of $A$ generated by $S$. Thus $\langle S \rangle =
\langle S^{\#} \rangle$ and coincides with the norm closure of $ASA$.

\begin{proposition}
\label{PropTROIdeals}
Let $e$ and $f$ be projections in a $C^*$-algebra $A$.
\begin{enumerate}[(a)]
\item If $I$ is an ideal in the TRO $eAf$, then $I =
e \langle I \rangle f = 
\langle I \rangle \cap e A f$;

\item if $J$ is an ideal of $A$, then $\langle e J e \rangle  =
\langle e J \rangle = \langle J e \rangle$;

\item there is a bijective correspondence between the ideals of $eA$
and of $eAe$ given by $I \leftrightarrow Ie$.
\end{enumerate}
\end{proposition}

\begin{proof}
\begin{enumerate}[(a)]
\item If $I$ is an ideal of $eAf$, then $e(AIA) f = (eAf) I^{\#} (eAf)
\subseteq I$ so that $e \langle I \rangle  f \subseteq I$, giving
(a).

\item
Since $(eJ)^{\#} = J e$ we have $\langle e J \rangle = \langle J e
\rangle$, which is the norm closure of $AeJ$ and of $J e A$. Since
$A(eJe)A = (AeJ)(JeA)$, the result follows after taking norm closures.

\item Let $J_1$ and $J_2$ be ideals of $A$ such that $eJ_1 e = e
J_2e$,
Using (a) and (b), $eJ_1 = \langle e J_1 \rangle \cap eA = \langle e
J_2 \rangle \cap eA = eJ_2$, which suffices.
\end{enumerate}
\end{proof}

\begin{proposition}
\label{MaximalOpSpIdeals}
Let $T$ be a TRO, $E$ a \JCstar-subtriple of $T$ and let $\mathscr{S}$
denote the set of ideals $I$ of $T$ such that $I \cap E = \{0\}$. Then
each element $I$ of $\mathscr{S}$ is contained in a maximal element of
$\mathscr{S}$.
\end{proposition}

\begin{proof}
Let $J \in \mathscr{S}$ and let $(I_\lambda)$ be a chain in
$\mathscr{S}$ with $J \subset I_\lambda$ for each $\lambda$
and consider the norm closure $I$ of $\bigcup I_\lambda$, an ideal of
$T$.
Given $x \in E$ and $z \in I$, choose $(z_n)$ in $\bigcup I_\lambda$
with $\|z_n - z \| \to 0$. Since, for each $\lambda$, $E \to
T/I_\lambda$ ($x \to x+I_\lambda$) is isometric, for each $n$ we have
$\|x\| \leq \|x - z_n\|$ implying that $\|x\| \leq \|x - z\|$ and
hence that $\|x\| = \|x + I\|$. Therefore $I \cap E = \{0\}$, and the
result follows from Zorn's lemma.
\end{proof}

A non-zero tripotent $u$ in a \JCstar-triple $E$ is said to be minimal
if $\tp u E u = \IC u$. If $E$ has a predual and contains a minimal
tripotent but has no non-trivial weak*-closed ideals, then it is a
Cartan factor. See \S\ref{SectionCartanFs} for a detailed discussion.

\section{Universal objects}
\label{SectionUniversal}

We shall examine a variety of universal objects. In order to avoid
later repetition we begin with general remarks. Suppose $\mathscr{C}$
is a subcategory of a category $\mathscr{D}$ and let $E \in \mathrm{Ob}
(\mathscr{D})$.
Let (by a slight abuse) $ (F, \alpha)$ be said to be a
universal $\mathscr{C}$-object for $E$, where $F \in \mathrm{Ob}
(\mathscr{C})$
and $\alpha \colon E \to F$ is a
$\mathscr{D}$-morphism, if each $\mathscr{D}$-morphism $\pi \colon
E \to G \in \mathrm{Ob}(\mathscr{C})$ entails the existence of a
unique $\mathscr{C}$-morphism $\tilde{\pi} \colon F \to G$ such that
$\tilde{\pi} \circ \alpha = \pi$.
In these circumstances $\tilde{\pi} = \id_F$ when $\pi = \alpha$. If
$(H, \beta)$ is another universal $\mathscr{C}$-object for $E$ then
there is a a (unique) $\mathscr{C}$-isomorphism
$\tilde{\beta} \colon F \to H$ with
$\tilde{\beta} \circ \alpha = \beta$, so that $(F, \alpha)$ and $(H,
\beta)$ are naturally equivalent. We may write $(F, \alpha) \equiv (H,
\beta)$ to signify this. If for each $E \in \mathrm{Ob}(\mathscr{D})$
a universal $\mathscr{C}$-object $(\mathscr{C}(E), \alpha_E)$ exists,
then there is a natural covariant functor $\mathscr{D} \to
\mathscr{C}$ sending $E$ to $\mathscr{C}(E)$ with corresponding (and
obvious) action on morphisms.

We shall establish the existence of the universal $C^*$-algebra and of
the universal TRO of a \JCstar-triple and shall discuss connections
with certain associated universal objects. We proceed by exploiting
the general construction \cite[II.8.3]{Blackadar} of the universal
$C^*$-algebra $C^*(\mathscr{G}, \mathscr{R})$ of a set of generators
$\mathscr{G}$ and relations $\mathscr{R}$ when the latter is
realisable among operators on a Hilbert space.

We remark that in Theorem~\ref{CstaruXexists} below, (a) is a formal
consequence of (b) as is the injectivity of $\alpha_E$. The same goes
for the statements of Corollary~\ref{TROuXexists}.

\begin{theorem}
\label{CstaruXexists}
Let $E$ be \JCstar -triple. Up to natural equivalence with
$*$-isomorphism, there is a unique
pair $(C^* (E), \utro E)$
where $C^* (E)$ is a $C^*$-algebra and
$\utro E \colon E \to C^* (E)$ is an injective triple homomorphism,
with the following properties:
\begin{enumerate}[(a)]
\item $\utro E(E)$ generates $C^* (E)$ as a $C^*$-algebra;
\item for each triple morphism $\pi \colon E \to
A$, where $A$ is  a $C^*$-algebra,
there is a unique
$*$-homomorphism $\tilde{\pi} \colon C^* (E) \to A$ with
$\tilde{\pi} \circ \utro E = \pi$.
\end{enumerate}
\end{theorem}

\begin{proof}
Consider a set $\mathscr{G} =
\{\alpha_a : a \in E\}$ of generators and a set
of relations $\mathscr{R}$
consisting of the union of the sets
$\{ \alpha_{\lambda a + b } - \lambda \alpha_a - \alpha_b : \lambda \in
\IC, a, b \in E \}$,
$\{ \alpha_{\tp{a}{ b}{ c}} - (\alpha_a \alpha_b^* \alpha_c + \alpha_c
\alpha_b^* \alpha_a)/2  : a,b,c \in E\}$
and
$ \{ \|\alpha_a\| \leq \|a\|: a \in E\}$.
Now put $C^* (E)
= C^*(\mathscr{G}, \mathscr{R})$
and $\utro E(a) = \alpha_a$ for each $a \in A$.
By construction, $\utro E(E)$ generates $C^* (E)$ and 
$\utro E \colon E\to C^*(E)$ is a triple homomorphism.

Let $\pi \colon E \to A$ be a triple homomorphism into a $C^*$-algebra
$A$. Since $\{ \pi(a) : a \in E \}$ satisfies the relations
$\mathscr{R}$, the universal property of $C^*(\mathscr{G},
\mathscr{R})$ implies the existence of $\tilde{\pi}$ as claimed in
(b). Moreover, since there is an example of an injective $\pi$, $\utro
E$ is injective. Uniqueness of $(C^*(E), \utro E)$ up to
$*$-isomorphism is clear.
\end{proof}

\begin{corollary}
\label{TROuXexists}
Let $E$ be a \JCstar -triple.
Up to natural equivalence with TRO isomorphism
there is a unique pair $(\TROu (E), \utro E)$
where $\TROu (E)$ is  a TRO and $\utro E \colon E \to \TROu (E)$
is an injective triple homomorphism
with the following properties:
\begin{enumerate}[(a)]
\item $\utro E(E)$ generates $\TROu (E)$ as a TRO;
\item
\label{TROuUniversalProp}
 for each triple morphism $\pi \colon E \to
T$, where $T$ is a TRO, there is a unique TRO
morphism $\tilde{\pi} \colon \TROu (E) \to T$ such that
$\tilde{\pi} \circ \utro E = \pi$.
\end{enumerate}
\end{corollary}

\begin{proof}
Let
$(C^* (E), \utro E)$ be as in Theorem~\ref{CstaruXexists} and let
$\TROu (E)$ denote the TRO generated by $\utro E(E) $ in $ C^*(E)$
(giving (a)).
Given a triple
homomorphism $\pi \colon E \to
T \subset \BH$, where $T$ is a TRO,
the $*$-homomorphism $C^*(E)$ to $\BH$ satisfying
Theorem~\ref{CstaruXexists} (b) restricts to the required TRO
homomorphism
$\tilde \pi \colon \TROu (E) \to T$, the uniqueness of which being
implied by (a).
\end{proof}

\begin{definition}
Let $E$ be a \JCstar-triple. In the notation of
Theorem~\ref{CstaruXexists} and Corollary~\ref{TROuXexists}, we define
$C^* (E)$ and $\TROu (E)$, more formally
$(C^* (E), \utro E)$ and $(\TROu (E), \utro E)$,
to be the \emph{universal $C^*$-algebra} and \emph{universal TRO} of
$E$. In each case, we refer to $\utro E$ as the universal embedding.
\end{definition}

\begin{remarks}
\label{functorialremarks}
We may define the \emph{universal \JCstar-algebra} $J^*(E)$
of a \JCstar triple $E$ to be the \JCstar-algebra generated by $\utro
E (E)$ in $C^* (E)$, with universal embedding $\utro E \colon E \to
J^*(E)$ in this case, by the method of proof of
Corollary~\ref{TROuXexists}. Each triple homomorphism $\pi \colon E
\to \BH$ induces a unique Jordan $*$-homomorphism $\tilde{\pi} \colon
J^*(E) \to \BH$ with $\tilde{\pi} \circ \utro E = \pi$.

If $T$ is a TRO and $A$ is a \JCstar-algebra, we may define
the universal $C^*$-algebras
$C^*_{\mathrm{TRO}}(T)$ and $C^*_{\mathrm{J}}(A)$
by the procedure of Theorem~\ref{CstaruXexists} with
transparent modifications to the set $\mathscr{R}$ of relations.
Thus, via the arising universal embedding of $T$ into
$C^*_{\mathrm{TRO}}(T)$, each TRO homomorphism $T \to \BH$ `lifts' to
a *-homomorphism $C^*_{\mathrm{TRO}}(T) \to \BH$. The corresponding
statement is true of $C^*_{\mathrm{J}}(A)$ in terms of Jordan
$*$-homomorphisms $A \to \BH$. 
Of these we note that $C^*_{\mathrm{J}}(A)$ has been studied in
detail in \cite{AlfsenHOSchultzActa1980,HancheOlsenJC} and
\cite[chapter 7]{HOS} in the guise of the universal $C^*$-algebras of
the $JC$-algebra $A_{sa}$ ($ = \{ a \in A : a^* =
a\}$).

For a \JCstar-triple $E$, all five universal objects introduced above
appear in the following single statement
\begin{equation}
\label{UnivOfUnivIsUniv}
C^*_{\mathrm{J}}(J^*(E)) = C^*(E) = C^*_{\mathrm{TRO}}(\TROu(E))
\end{equation}
with universal embeddings given by the inclusions
$J^*(E), \TROu(E) \subset C^*(E)$.

Indeed, $J^*(E)$ generates $C^*(E)$ since $\utro E(E)$ does. Given a
Jordan homomorphism $\pi \colon J^*(E) \to \BH$, the universal
property of $\utro E \colon E \to C^*(E)$ guarantees a
$*$-homomorphism $\psi \colon C^*(E) \to \BH$ agreeing with $\pi$ on
$\utro E (E)$. Since $J^*(E)$ is a \JCstar-algebra generated by $\utro
E(E)$ , $\psi$ must agree with $\pi$ on $J^*(E)$. The remaining claim
has an analogous proof.

It follows from (\ref{UnivOfUnivIsUniv}) together with
\cite[Theorem 7.1.8]{HOS} that there is a unique $*$-antihomomorphism
$\Phi_0$, of $C^*(E) = C^*_{\mathrm{J}}(J^*(E))$ acting identically on $J^*(E)$.
Since $\Phi_0$ has order 2 and since given $y = x_1 x_2^* \cdots
x_{2n}^* x_{2n+1}$ with $x_1,  \ldots, x_{2n+1} \in \utro E(E)$,
we have $\Phi_0(y) = x_{2n+1} x_{2n}^* \cdots x_2^* x_1 \in \TROu(E)$,
it follows that $\Phi_0(\TROu(E)) = \TROu(E)$. Thus, the restriction
$\Phi$ of $\Phi_0$ to $\TROu(E)$ is a TRO antiautomorphism of order 2.
If $\psi \colon \TROu(E) \to \TROu(E)$ is any TRO antiautomorphism
of $\TROu(E)$ with $\psi \circ \utro E = \utro E$, then $\psi \circ
\Phi$ is a TRO automorphism of $\TROu(E)$ acting identically on $\utro
E(E)$, so that $\psi = \Phi^{-1} = \Phi$. We summarise these remarks
in the following.
\end{remarks}

\begin{theorem}
\label{canonicalinvolution}
Let $E$ be a \JCstar -triple. Then
\begin{enumerate}[(a)]
\item
there is a unique
$*$-antiautomorphism $\Phi_0$ of $C^*(E)$ acting identically on
$\utro E(E)$;
\item
the restriction $\Phi$ of $\Phi_0$ to $\TROu(E)$
is the unique TRO antiautomorphism of $\TROu(E)$ acting identically on
$\utro E(E)$;
\item
\label{PhiInvol}
$\Phi_0$ and $\Phi$ have order $2$;
\item
\label{PhiFixedJstarX}
$\Phi$ acts identically on $J^*(E)$;
\item
\label{RightLeftAlgsExchanged}
$\Phi(\mathscr{L}_T) = \mathscr{R}_T$,
where $T = \TROu(E)$.
\end{enumerate}
\end{theorem}

We refer to $\Phi_0$ and $\Phi$ of Theorem~\ref{canonicalinvolution} 
as the \emph{canonical involutions} of $C^*(E)$ and $\TROu(E)$,
respectively.

\begin{proposition}
\label{MsummandsTheorem}
If a \JCstar -triple $E$ is the sum of orthogonal ideals
$I$ and $J$,
then $(\TROu(E), \utro E) \equiv
(\TROu(I) \oplus \TROu(J), \utro I \oplus
\utro J)$.
\end{proposition}

\begin{proof}
This follows from Lemma~\ref{MsummandsProp} and the universal property
of $\TROu(\cdot)$.
\end{proof}

\begin{proposition}
\label{TROuOfJCalg}
Let $A$ be a \JCstar-algebra. Then 
\[
(\TROu(A), \utro A) = (C^*_{\mathrm{J}}(A),
\beta_A)
\]
where $\beta_A \colon A \to C^*_{\mathrm{J}}(A)$ is the universal
\JCstar-algebra embedding. In this identification, the canonical
involution of $\TROu(A)$ is the involutory $*$-antiautomorphim of
$C^*_{\mathrm{J}}(A)$ acting identically on $\beta_A(A)$.
\end{proposition}

\begin{proof}
By Lemma~\ref{TROofAlgIsAlg} we have $\TRO(\beta_A(A)) =
C^*_{\mathrm{J}}(A)$.
Given a triple homomorphism $\pi \colon A \to \BH$, consider its
weak*-continuous extension $\psi \colon A^{**} \to \BH$ and put $u =
\psi(1)$. Then $\pi \colon A \to \BH_2(u)$ is a Jordan
$*$-homomorphism into the $C^*$-algebra $\BH_2(u)$, and so induces a
$*$-homomorphism $\tilde{\pi} \colon C^*_{\mathrm{J}}(A) \to \BH_2(u)$ with
$\tilde{\pi} \circ \beta_A = \pi$. By Lemma \ref{TROPeirce2} (a)
$\tilde{\pi} \colon C^*_{\mathrm{J}}(A) \to \BH$ is a TRO homomorphism, as
required.
\end{proof}

\section{Reversibility}
\label{SectionReversibility}

By definition (see \cite[2.3.2]{HOS}, for example) a \JCstar-subalgebra
$A$ of a $C^*$-algebra $B$ is called \emph{reversible} in $B$ if
\[
a_1 a_2  \cdots a_{n}
+
a_n  \cdots a_{2} a_{1} \in A
\]
whenever $a_1, \ldots, a_n \in A$. An equivalent condition is
\[
a_1 a_2^* a_3 \cdots a_{2n}^* a_{2n+1}
+
a_{2n+1} a_{2n}^*  \cdots a_{2}^* a_{1} \in A
\]
whenever $a_1, \ldots, a_{2n+1} \in A$,
since if $A$ satisfies the latter condition in $B$, then so does
$A^{**}$ in $B^{**}$ implying (since $A^{**}$ has an identity) that
$A^{**}$ is reversible in $B^{**}$ and hence that $A = A^{**} \cap B$
is reversible in $B$.

A \JCstar-algebra is said to be \emph{universally reversible} if its
canonical image in $C^*_{\mathrm{J}}(A)$ is reversible \cite{HancheOlsenJC} (the
reversibility of $A$ is equivalent to that of $A_{sa}$).

We shall now introduce the notion of reversibility for
\JCstar-triples.

\begin{definition}
\label{reverseibilityDef}
A \JCstar -subtriple $E$ of a TRO $T$ is said to be \emph{reversible}
in $T$ if 
\[
a_1 a_2^* a_3 \cdots a_{2n}^* a_{2n+1}
+
a_{2n+1} a_{2n}^* \cdots a_{2}^* a_{1} \in E
\]
whenever $a_1, \ldots, a_{2n+1} \in E$

We say that a \JCstar -triple $E$ is \emph{universally reversible} if
$\utro E(E)$ is reversible in $\TROu(E)$.
\end{definition}

By this definition and the preamble, the reversibility of a
\JCstar-algebra is equivalent to its reversibility as a \JCstar-triple
and, via Proposition~\ref{TROuOfJCalg}, a \JCstar-algebra is
universally reversible as a \JCstar-triple if and only if it is
universally reversible as a \JCstar-algebra.

Following \cite{HancheOlsenJC}, we shall proceed to obtain a serviceable
characterisation of the universal TRO of a universally reversible
\JCstar-triple.

\begin{lemma}
\label{canonicalinvolutionTRO}
The following are equivalent for a \JCstar-triple $E$.
\begin{enumerate}[(a)]
\item $E$ is universally reversible;
\item
if $\pi \colon E \to T$ is a triple homomorphism into a TRO, then
$\pi(E)$ is reversible in $T$;
\item 
\label{FPpropPhi}
$\utro E(E) = \{ a \in \TROu(E) : \Phi(a) = a \}$ (where $\Phi$ is the
canonical involution of $\TROu(E)$).
\end{enumerate}
\end{lemma}

\begin{proof}
The implication (a) $\Rightarrow$ (b) is a simple consequence of the
universal property, (c) $\Rightarrow$ (a) is immediate from the fact
that $\Phi$ fixes each point of $\utro E (E)$ and (b) $\Rightarrow$
(a) is clear.
Assume (a). Then the condition $b + \Phi(b) \in \utro E (E)$ holds for
all $b$ of the form $a_1 a_2^* \cdots a_{2n}^* a_{2n+1}$ with the
$a_i \in \utro E (E)$, and hence holds for all $b \in \TROu(E)$.
Thus if $b \in \TROu(E)$ with $\Phi(b) = b$, then
$b = (b + \Phi(b))/2 \in \utro E(E)$, proving (c).
\end{proof}

\begin{lemma}
\label{PhiInvIdeals}
Let $E$ be a universally reversible \JCstar -triple and
$\mathcal{I}$ an  ideal of $ \TROu(E)$
such that
$\utro E(E) \cap \mathcal{I}= \{0\}$
and $\Phi(\mathcal{I}) = \mathcal{I}$. Then $\mathcal{I} = \{0\}$.
\end{lemma}

\begin{proof}
Let $b \in \mathcal{I}$. Then $b + \Phi (b) \in \utro E(E) \cap
\mathcal{I}$ so that $\Phi(b) = -b$.
Thus for $a \in \utro E (E)$, since $\tp aab  \in I$ we have
$-\tp aa b = \Phi(\tp aab) = \tp aab $ and hence that $b^* a = 0$. It
follows that $b^*\TROu(E) = \{0\}$, giving $b = 0$.
\end{proof}

\begin{theorem}
\label{ThmUnivRev}
Let $E$ be a universally reversible \JCstar -triple, let
$\pi \colon E \to \BH$ be an injective triple
homomorphism
and let $\tilde{\pi} \colon \TROu(E) \to \BH$ be the TRO homomorphism
so that $\tilde{\pi} \circ \utro E = \pi$. Suppose
there is a TRO
antiautomorphism $\psi$ of $\TRO(E)$ so that $\psi \circ \pi = \pi$. Then
$\tilde{\pi}$ is a isomorphism. (Hence,
$(\TROu(E), \utro E) \equiv (\TRO(\pi(E)), \pi)
$ with canonical involution $\psi$.)
\end{theorem}

\begin{proof}
Let $\mathcal{I} = \ker \tilde{\pi}$.
We have $\utro E (E) \cap \mathcal {I} = \{0\}$ since $\pi$ is
injective. By assumption, the TRO homomorphism $\psi \circ \tilde{\pi}
\circ \Phi$ agrees with $\tilde{\pi}$ on $\utro E (E)$ and hence on
$\TROu(E)$, giving $\psi \circ \tilde{\pi} = \tilde{\pi} \circ \Phi$.
It follows that $\Phi(\mathcal{I}) \subseteq \mathcal{I}$. Hence
$\mathcal{I} = \{0\}$ by Lemma~\ref{PhiInvIdeals}.
\end{proof}

\begin{corollary}
\label{TROuUnivRevT}
Let $T$ be a universally reversible TRO in a $C^*$-algebra $A$.
Suppose $T$ has no ideals of codimension one and there is a TRO
antiautomorphism $\theta \colon A \to A$ of order $2$. Then $\TROu(T)
= T \oplus \theta(T)$ with universal embedding $a \mapsto a \oplus
\theta(a)$.
\end{corollary}

\begin{proof}
We have an injective triple homomorphism $\pi \colon T \to T \oplus
\theta(T)$ ($a \mapsto a \oplus \theta(a)$), a TRO antiautomorphism
$\psi \colon T \oplus \theta(T)$ ($a \oplus \theta(b) \mapsto b \oplus
\theta(a)$), such that $\psi \circ \pi = \pi$. Put $E = \pi(T)$. By
Theorem~\ref{ThmUnivRev}, it is enough to show $\TRO(E) = T \oplus
\theta(T)$.

Given $a, b ,c \in T$, we have $(\trop a b c - \trop c b a , 0) =
\trop{\pi(a)}{\pi(b)}{\pi(c)} - \pi(\trop c b a ) \in \TRO(E)$.
Thus $S \oplus \{0\} \subset \TRO(E)$ where $S$ is the set $\{ \trop a
b c - \trop c b a : a, b, c \in T\}$. The norm closure of all TRO
products $[a_1 \ldots a_{2n+1}]$ ($= a_1 a_2^* \cdots a_{2n}^*
a_{2n+1}$), where the $a_i \in T$ and at least one of them belongs to
$S$, is the ideal $\mathcal{J}$ of $T$ generated by $S$. Since
\[
[a_1 \ldots a_{2n+1}] \oplus 0 = [ \pi(a_1) \cdots (a_i \oplus 0)
\cdots \pi(a_{2n+1})] \in \TRO(E)
\]
whenever $a_i \in S$ with $1 < i < 2n+1$, and correspondingly when
$i = 1$ or $i = 2n+1$, we have that
$\mathcal{J} \oplus \{0\} \subseteq \TRO(E)$. Suppose $\mathcal{J} \neq T$.
We may choose a non-trivial TRO homomorphism $\phi \colon T \to \BH$
vanishing on $\mathcal{J}$, and then $\phi(x)\phi(y)^* \phi(z) =
\phi(z) \phi(y)^* \phi(x)$ for $x, y, z \in T$.
It follows that (see \cite[Proposition 6.2]{KaupFibreBund}) $\phi(T)$
has an ideal of codimension one, as therefore does $T$, a
contradiction. Hence $\mathcal{J} = T$. Therefore $T \oplus \theta(T)
= \TRO(E)$.
\end{proof}

\section{Cartan factors}
\label{SectionCartanFs}

The Cartan factors in question are the \JWstar-triple factors
possessing minimal tripotents. (The two exceptional factors of general
\JBstar-triple theory are beyond our scope.) An exhaustive analysis
may be found in \cite{DangFriedman1987}. A measure of their
significance is that \JCstar-triples may be realised as a weak*-dense
subtriple of an $\ell_\infty$ sum of Cartan factors \cite{FRgn}.
The purpose of
this section is to obtain concrete descriptions of the universal TROs
of Cartan factors and to determine which Cartan factors are
universally reversible.

Given a Hilbert space $H$ we may define a transposition, $x \mapsto
x^t$, on $\BH$ by $x^t(\xi) = \overline{x^*\left(\bar{\xi} \right)}$ where $\xi
\to \bar \xi$ is a conjugation on $H$. By the dimension, $\dim (H)$,
of $H$ we shall mean the cardinality of an orthonormal basis of $H$,
allowing infinite cardinals.

There are four types of Cartan factors, named \emph{rectangular},
\emph{hermitian}, \emph{symplectic} and \emph{spin factors}, which up
to linear isometry may be realised in the following concrete forms:
\begin{description}
\item[\emph{rectangular}, $R_{m,n}$]
$\BH e$ where $e$ is a projection in $\BH$ with $m = \dim (eH) \leq
\dim (H) = n$

\item[\emph{hermitian}, $S_n$]
$\{x \in \BH : x^t = x \}$ for $\dim (H) =n \geq 2$

\item[\emph{symplectic}, $A_n$]
$\{x \in \BH : x^t = -x \}$ for $\dim(H) = n \geq 4$

\item
[\emph{spin}, $\Sp(n)$]
the norm closed linear span, in a unital $C^*$-algebra,
of the identity and a \emph{spin
system} $\{s_i : i \in I\}$ of cardinality $n \geq 2$ of
anti-commuting symmetries (self adjoint unitaries with $s_is_j +
s_js_i = 0$ if $i \neq j$).
\end{description}

Other than the fact that $R_{2,2}$, $S_2$ and $A_4$ are spin factors
(of dimensions $4$, $3$ and $6$ respectively) the types described
above are mutually exclusive. The \emph{rank} of a Cartan factor is
the cardinality of a maximal orthogonal family of minimal tripotents
it contains. By definition, Hilbert spaces are the Cartan factors of
rank 1. Spin factors, $R_{m,n}$ and $S_n$ have rank 2, $m$ and $n$,
respectively, and $A_n$ has rank $\left[ \frac n 2 \right]$ when $n <
\infty$ and rank $n$ otherwise. The reflexive Cartan factors are those
of finite rank.

We shall proceed on a case by case basis roughly distinguished by rank
and pathology.

\subsection*{Hilbert spaces}

Let $H$ be a Hilbert space. Our approach involves the CAR algebra
$\CAR (H)$ acting upon antisymmetric Fock space $\Fock(H)$ brief
details of which \cite{KadisonRingroseII, KadisonRingroseIV, HOS}
are rehearsed below.

We recall that $\Fock(H)$ is the $\ell_2$ direct sum
$\bigoplus_{m=0}^\infty \Lambda^m H$, where $\Lambda^0 H = \IC$ and
for each $m \geq 1$, $\Lambda^m H$ is the Hilbert space of normalised
antisymmetric tensors given by
\[
\xi_1 \wedge \xi_2 \wedge \cdots \wedge \xi_m
= (m!)^{1/2} P_m (\xi_1 \otimes \cdots \otimes \xi_m),
\]
$P_m$ being the projection on $\bigotimes^m H$ given by $P_m =
\frac 1{m!}\sum_{\sigma \in S_m} \sgn (\sigma) u_\sigma$, where $S_m$ is the
symmetric group on $m$ letters and $u_\sigma$ is the unitary
determined by
\[
u_\sigma( \xi_1 \otimes \cdots \otimes \xi_m)
=
\xi_{\sigma(1)} \otimes \cdots \otimes \xi_{\sigma(m)}.
\]

Let $\{ e_i : i \in I\}$ be an orthonormal basis. With respect to a
linear ordering on $I$ $\{ e_{i_1} \wedge \cdots \wedge e_{i_m}: i_1 <
\cdots < i_m \}$ is an orthonormal basis of $\Lambda^m H$. When $H$
has finite dimension $n$, $\Lambda^m H = \{0\}$ for all $m > n$ and
has dimension $\binom  n m$ whenever $m \leq n$, and $\Fock(H)$ has
dimension $2^n$. In general, for each $\xi \in H$ the annihilation
operator $a(\xi) \in \mathcal{B}(\Fock(H))$ (the dual of the creation
operator $c(\xi)$ determined by $c(\xi)\alpha = \xi \wedge \alpha$, $\alpha
\in \Lambda^m H$, $m \geq 0$) depends antilinearly on $\xi$ and
satisfies 
\[
\|a(\xi)\| = \|\xi\|, \quad
a(\xi) (\Lambda^{m+1} H) \subset \Lambda^{m} H
\qquad (\mbox{for all }\xi \in H \mbox{ and } m \geq0),
\]
together with the canonical anticommutation relations (CAR)
\[
a(\xi)a(\eta) + a(\eta)a(\xi) = 0, \quad
a(\xi)a(\eta)^* + a(\eta)^*a(\xi) = \langle \eta, \xi \rangle \id,
\]
(for all $ \xi, \eta \in H$),
where $\id$ is the identity element in $\mathcal{B}(\Fock(H))$.
Writing $a_i = a(e_i)$ for each $i \in I$ the CAR specialise to
\[
a_i a_j + a_j a_i =0, \quad a_i a_j^* +  a_j^* a_i = \delta_{(i,j)}
\id, \qquad \mbox{for all } i, j \in I.
\]
Put $\tilde{H} = \{ a(\xi) : \xi \in H\}$. The $C^*$-algebra generated
by $\tilde{H}$ is the CAR algebra $\CAR(H)$, over $H$. Letting $\xi
\mapsto \bar \xi$ denote the conjugation determined by $\{ e_i : i \in
I\}$ we have that the map
\[
\psi \colon H \to \CAR(H) \quad (\xi \mapsto a\left( \bar \xi \right))
\]
is a triple isomorphism onto the \JCstar-subtriple $\tilde{H}$ of
$\CAR(H)$.

Retaining these notations, we shall show that the TRO generated by
$\tilde{H}$ in $\CAR(H)$ is $\TROu(H)$ and that $\psi = \utro H$.

\begin{theorem}
\label{TROuHTheorem}
$\TROu(H)$ is  TRO generated by
$\tilde{H}$ in $\CAR(H)$ with $\utro H$ given by $\utro H(\xi) =
a\left( \bar \xi \right)$. If $H$ has finite dimension, $n$, then
$\TROu(H)$ coincides with the $\ell_\infty$ sum
$\bigoplus_{m=0}^{n-1} \mathcal{B} \left( \Lambda^{m+1} H, \Lambda^m
H\right)$. In addition, $H$ is universally reversible if and only if
$\dim(H) \leq 2$.
\end{theorem}

\begin{proof}
Let $\pi \colon H \to A$ be a triple homomorphism into a $C^*$-algebra
$A$ and put $T = \TRO(\pi(H))$. Then $\pi(H)$ is a Hilbert space with
orthonormal basis $\{x_i : i \in I\}$ where $x_i = \pi(e_i)$ for each $
i \in I$. Given $i$, $j$ and $k$ in $I$ we have the \emph{rules}
\[
x_i x_j^*x_k + x_k x_j^* x_i = 2 \tp{x_i}{x_j}{x_k} =
\delta_{(j,k)}x_i + \delta_{(i,j)}x_k,
\]
so that
\[
x_i x_j^*x_i = 0 \mbox{ and } x_i x_j^*x_j = - x_j x_j^* x_i + x_i
\mbox{ whenever } i \neq j,
\]
and
$x_i x_j^* x_k = - x_k x_j^* x_i$ whenever $i \neq j \neq k$.

To prove the statement we shall suppose, first, that $H$ has finite
dimension $n$ with $I = \{ 1, \ldots, n\}$ (given the standard
ordering). Then $T$ is the norm closed linear span of TRO products of
the form
\begin{equation}
\label{Forma}
x_{i_1} x_{j_1}^* \cdots x_{j_m}^* x_{i_{m+1}}, 
\end{equation}
the indices $ i_r$ and $ j_s $ ranging over $ 1,
\ldots, n$.

Via the above \emph{rules} appropriate shuffling of TRO products reveals
that $T$ is linearly generated by the products of the form
(\ref{Forma}) where
\begin{equation}
\label{Formb}
1 \leq i_1 < \cdots < i_{m+1} \leq n, \quad
1 \leq j_1 < \cdots < j_m \leq n.
\end{equation}
Formally, there are $\sum_{m=0}^{n-1} \binom{n}{m+1} \binom{n}{m}$
such products. In the special case  of $\TRO(\tilde{H})$ we note
that the products
\begin{equation}
\label{Formc}
a_{i_1} a_{j_1}^* \cdots a_{j_m}^* a_{i_{m+1}}, \quad
\mbox{where the indices satisfy (\ref{Formb})}
\end{equation}
are linearly independent (as follows from the linear independence
of the corresponding Wick-ordered products
$a_{j_1}^* \cdots a_{j_m}^* a_{i_1}  \cdots a_{i_{m+1}}$ in
$\CAR(H)$ \cite[10.5.88, 12.4.40]{KadisonRingroseIV}).

Since all such products lie in
$\bigoplus_{m=0}^{n-1} \mathcal{B} \left( \Lambda^{m+1} H, \Lambda^m
H\right)$,
a dimension count shows that this space is exactly
$\TRO(\tilde{H})$. To establish the universal property, consider the
linear map $\tilde{\pi} \colon \TRO(\tilde{H}) \to T$ such that
\begin{equation}
\label{HilbSpPiTilde}
\tilde{\pi} (a_{i_1} a_{j_1}^* \cdots a_{j_m}^* a_{i_{m+1}})
=
x_{i_1} x_{j_1}^* \cdots x_{j_m}^* x_{i_{m+1}}
\end{equation}
whenever the indices satisfy (\ref{Formb}).
Since the $a_i$ and $x_i$ formally satisfy the same TRO relations, as
the above \emph{rules} show, $\tilde{\pi}$ is a TRO homomorphism satisfying
the requirements. This settles the case $I$ finite.

Suppose now that $I$ is infinite (and linearly ordered as in the
preamble). It follows from the finite case settled above that the
linear map $\tilde{\pi}$ defined on the (algebraic) linear span of
$\{ a_{i_1} a_{j_1}^* \cdots a_{j_m}^* a_{i_{m+1}} : \mbox{ the
indices satisfy (\ref{Formb})}, m \geq 0\}$
by the formula
(\ref{HilbSpPiTilde})
obeys the TRO rules $\tilde{\pi}(\trop x y  z) = \trop{\tilde{\pi}(x)}
{\tilde{\pi}(y)} {\tilde{\pi}(z)}$ and (hence) is contractive. Thus
$\tilde{\pi}$ extends to a TRO homomorphism from $\TRO(\tilde{H})$
onto $T$, as required.

We shall now turn to the question of universal reversibility.
If the dimension of $H$ does not exceed 2, then it is easy to see that
$H$ is reversible in $\TROu(H)$.
On the other hand if $H$ has dimension greater than 2, consider $y = x
+ \Phi(x)$ where $x =
a_{i_1} a_{i_1}^* a_{i_2} a_{i_2}^* a_{i_3}$ with the indices
belonging
to $I$ such that $i_1 < i_2 < i_3$. By repeated application of the
CAR,
\begin{equation}
\label{HnotRevEqn}
y
= 2 x
-  a_{i_1} a_{i_1}^* a_{i_3}
-  a_{i_2} a_{i_2}^* a_{i_3}
+ a_{i_3}
\end{equation}
which lies in $\tilde{H}$ if the latter
is reversible in $\TRO(\tilde{H})$.
In which case, since $a_i^* y a_i^* = 0$ for all $i \neq i_3$, the
right hand side of (\ref{HnotRevEqn}) belongs to $\IC a_{i_3}$
contradicting the linear independence of the terms involved.
\end{proof}

\subsection*{Spin factors}
In view of \cite[\S6.2]{HOS}, \cite[\S9.3]{Pisier}
and Proposition~\ref{TROuOfJCalg} the universal
TRO of a spin factor is essentially known.
Let $V$ be a spin factor generated by $1$ and a  spin system indexed
by $I$ in a $C^*$-algebra $A$ and let $A_V$ denote the $C^*$-algebra
generated by $V$ in $A$. We assume $I$ has cardinality at least 2.
Since $V$ is a \JCstar-subalgebra  of $A$, by
Proposition~\ref{TROuOfJCalg} together with the discussion in
\S\ref{SectionReversibility} we have:

\begin{lemma}
\label{SpinsRevLemma}
$(\TROu(V), \utro V) = (C^*_{\mathrm{J}}(V), \beta_V)$, and $V$ is universally
reversible (as a \JCstar-triple) if and only if $\dim(V) \leq 4$.
\end{lemma}

If the dimension of $V$ is odd (\emph{i.e.} if $I$ has even
cardinality) or infinite then $C^*_{\mathrm{J}}(V)$ is (*-isomorphic to) the CAR
algebra over $\ell_2(I)$ so that $(\TROu(V), \utro V) \equiv (A_V, V
\hookrightarrow A_V)$ and $A_V \cong \CAR(\ell_2(I))$. Thus, when $V$
has odd or infinite dimension every concrete \JCstar-triple
embedding of $V$ is universal. On the other hand when $I$ has
cardinality $2n+1 < \infty$, $V$ has universal embedding in
$M_{2^{n+1}}(\IC)$ with $\TROu(V) = M_{2^n}(\IC) \oplus M_{2^n}(\IC)$,
so that the restrictions of the natural projections $M_{2^n}(\IC)
\oplus M_{2^n}(\IC) \to M_{2^n}(\IC)$ to $\utro V(V)$ cannot be
universal.

In order to provide details sufficient for operator space
considerations (see \S\ref{secOpSpStructs}) we shall recall the
standard representations of finite dimensional spin factors
\cite[\S6]{HOS}. Given $x \in M_2(\IC)$ and $n \geq 1$ let the
$n$-fold tensor product $x \otimes \cdots \otimes x \in M_{2^n}(\IC)$
be denoted $x^n$ and let $x^0 \otimes y = y \otimes x^0 = y$ for all
$y \in M_{2^n}(\IC)$. Let 
$\sigma_1$, $\sigma_2$ and $\sigma_3$ denote the Pauli spin matrices
$\begin{pmatrix}
1 & 0\\ 0 & -1
\end{pmatrix}$,
$\begin{pmatrix}
0 & 1\\ 1 & 0
\end{pmatrix}$,
$\begin{pmatrix}
0 & i\\ -i & 0
\end{pmatrix}$,
respectively.

Fixing $n \geq 1$ and putting $1 = 1_2^n$, we shall write
\[
V_{2n} = \SPAN \{ 1, s_1, \ldots, s_{2n} \}
\mbox{ and }
V_{2n+1} = \SPAN \{ V_{2n} \otimes 1_2 , s_{2n+1} \}
\]
where
$s_1 = \sigma_1 \otimes
1_2^{n-1}$, $ s_2 = \sigma_2 \otimes 1_2^{ n-1}$,
$ s_3 = \sigma_3 \otimes \sigma_1 \otimes 1_2^{ n-2}$,
$ s_4 = \sigma_3 \otimes \sigma_2 \otimes 1_2^{ n-2}$,
\dots, $s_{2n-1} = \sigma_3^{ n-1} \otimes \sigma_1$,
$ s_{2n} = \sigma_3^{ n-1} \otimes \sigma_2$,
and
$s_{2n+1} = \sigma_3^n \otimes \sigma_1$,
noting that $\{ s_1, \ldots, s_{2n} \}$ and
$\{ s_1 \otimes 1_2, \ldots, s_{2n} \otimes 1_2, s_{2n+1} \}$
are spin systems in $M_{2^n}(\IC)$ and $M_{2^{n+1}}(\IC)$
respectively.
With universal embeddings given by the inclusions
$V_{2n} \hookrightarrow M_{2^n}(\IC)$
and
$V_{2n+1} \hookrightarrow M_{2^{n+1}}(\IC)$ the upshot is that
\[
\TROu(V_{2n}) = M_{2^n}(\IC)
\mbox{ and } 
\TROu(V_{2n+1})
= M_{2^n}(\IC) \otimes D_2
= M_{2^n}(\IC) \oplus M_{2^n}(\IC),
\]
where $D_2$ is the diagonal subalgebra of $M_2(\IC)$.

Putting $t_{2n+1} = \sigma_3^n$ and noting that 
$\{ s_1, \ldots, s_{2n}, t_{2n+1}\}$ is a spin system in
$M_{2^n}(\IC)$, we shall write
\[
\tilde{V}_{2n+1} = \SPAN \{ 1, s_1, \ldots, s_{2n}, t_{2n+1}\}.
\]
We define linear maps
$\beta_n \colon \tilde{V}_{2n+1}  \to  \tilde{V}_{2n+1}$ by
\[
\beta_n(x) = x \mbox{ if } x \in V_{2n}, \quad \beta_n(t_{2n+1})
= - t_{2n+1};
\]
$\mu_n \colon \tilde{V}_{2n+1}  \to  V_{2n+1}$ by
\[
\mu_n(x) = x \otimes 1_2
\mbox{ if } x \in V_{2n}, \quad \mu_n(t_{2n+1}) =  s_{2n+1};
\]
$\gamma_n \colon M_{2^{n+1}}(\IC) \to M_{2^{n+1}}(\IC)$ by
\[
\gamma_n(x) = s x s,
\]
where $s = 1_2^n \otimes \sigma_2$, and we note that:
\begin{lemma}
\label{OddSpinsLemma}
\begin{enumerate}[(a)]
\item $\beta_n$ and $\gamma_n$ are Jordan isomorphisms.
\item $\gamma_n$ is a *-isomorphism of $M_{2^{n+1}}(\IC)$ with
$\gamma_n(V_{2n+1}) = V_{2n+1}$
such that $\gamma_n$ acts
identically on $V_{2n} \otimes 1_2$
and
$\gamma_n(s_{2n+1}) = - s_{2n+1}$.
\item
$\tilde{V}_{2n+1} \to M_{2^n}(\IC) \oplus M_{2^n}(\IC)$
($x \mapsto x \oplus \beta_n(x)$) is the universal embedding of
$\tilde{V}_{2n+1}$ into
$\TROu(\tilde{V}_{2n+1})$.
\end{enumerate}
\end{lemma}

\subsection*{Rectangular factors, rank $\geq 2$}

\begin{theorem}
\label{TROuRmn}
Let $E = \BH e$ where $e$ is a projection in $\BH$ of rank $\geq 2$.
Then $E$ is universally reversible with $\TROu(E) = E \oplus E^t$ and
$\utro E(x) = x \oplus x^t$. It may be supposed that $e^t = e$.
\end{theorem}

\begin{proof}
To ease notation we shall denote $\utro E$ by $\utr$ throughout. We
separate two cases.
\begin{enumerate}[(a)]
\item
Suppose $e$ has finite rank $m$.

Choose an orthonormal basis $\{ h_j : j \in J \}$ of $H$ so that
$\{h_i : i \in I\}$ is a basis of $eH$, $I$ being a subset of $J$ of
cardinality $m$, and denote the standard matrix units $h_k \otimes
h_i$ of $\BH$ by $E_{k, i}$ as $(k, j)$ ranges over $J \times J$.

Let $(j, i) \in J \times I$. Our first claim is that, for all $r \in
I$ with $r \neq i$, the elements $e_{j,i}$ and $f_{i, j}$ of
$\TROu(E)$ given by
\[
e_{j, i} = \trop {\utr (E_{j,i})} { \utr(E_{j,r})} {\utr (E_{j, i})},
\quad
f_{i, j} = \utr (E_{j,i}) - e_{j, i}
\]
are well-defined.
To see this take any subset $K \subset J$ of cardinality $m$ and let
$E_K$ denote the linear span of $\{ E_{s,r} : (s, r) \in K \times
J\}$. The latter is a linear isometric copy of $M_m(\IC)$ and hence
(by 
Proposition~\ref{TROuOfJCalg} and
\cite[7.4.15]{HOS} or \cite[Corollary 4.5]{HancheOlsenJC}) there
is a TRO homomorphism $\pi_K \colon E_K \oplus E_K^t \to \TROu(E)$
such that $\pi_K(x \oplus x^t) = \utr(x)$ for all $x \in E_K$.
For any $s \in K$ and $r \in I$ with $r \neq i$, it follows that
\[
\trop{\utr (E_{s,r})} { \utr(E_{s,r})} {\utr (E_{s, i})}
= \pi_K(E_{s,i} \oplus 0).
\]
Thus $e_{j,i}$ and $f_{i,j}$ are well-defined with $e_{j, i} =
\pi_K(E_{j,i} \oplus 0)$ and $f_{i,j} = \pi_K(0 \oplus E_{i,j})$
whenever $K$ is chosen so that $j \in K$. (We note that $E_{i,j} =
E_{j,i}^t = E_{j, i}^*$ for $(j, i) \in J \times I$.)

We next claim that the $e_{j, i}$ satisfy the same  TRO relations as
the $E_{j, i}$. Indeed, let $j, k, s \in J$ and $i, \ell, r \in I$. If
$\{ j, k, s\}$ has cardinality $\leq m$, \emph{a fortiori} if $3 \leq
m$, then the above implies (by a suitable choice of $K \subseteq J$)
\[
\trop{e_{j,i}}{e_{k,\ell}}{e_{s,r}}
= \delta_{(i, \ell)} \delta_{(k,s)} e_{j,r}.
\]
If $m =2$, we conclude that
\begin{eqnarray*}
\trop{e_{j,i}}{e_{k,\ell}}{e_{s,r}} 
&=& \trop {e_{j,i}}
{\trop  {e_{k,\ell}} {e_{k,\ell}} {e_{k,\ell}}  }
{ \trop{e_{s,r}}{e_{s,r}}{e_{s,r}} }
\\
&=& \trop{ \trop{e_{j,i}}{e_{k,\ell}}{e_{k,\ell}} }
{ \trop{e_{s,r}}{e_{s,r}}{e_{k,\ell}} } {e_{s,r}}\\
&=& \delta_{(i, \ell)} \delta_{(k,s)} e_{j,r}.
\end{eqnarray*}
By similar methods, we can also see that the $f_{i,j}$ satisfy the
same  TRO relations as the $E_{i, j}$ and are orthogonal to the $e_{k,
\ell}$.

Since the linear span of $\{ E_{j, i} : (j, i) \in J \times I\}$ is
norm
dense in $E$, it follows that there is a TRO isomorphism $\pi \colon E
\oplus E^t \to \TROu(E)$ sending each $E_{j, i} \oplus E_{r,s}^t$ to
$e_{j,i} + f_{r,s}$, in particular sending $E_{j, i} \oplus
E_{j,i}^t$ to $\utr(E_{j,i})$. Thus $\TROu(E)$ and $\utr$ may be
identified as in the statement.

The corresponding canonical involution is given by $x \oplus y^t
\mapsto y \oplus x^t$, the set of fixed points being $\{ x \oplus x^t
: x \in E\}$, proving that $E$ is universally reversible.

\item
Let $eH$ be infinite dimensional.

Consider $x_1, \ldots , x_{2k+1} \in E$. Since the dimension of the
Hilbert subspace generated by $x_1H + \cdots + x_{2k+1}H$ cannot exceed
that of $eH$ there is a projection $f$ in $\BH$ such that $eH$ and
$fH$ have the same dimension and $x_1, \ldots, x_{2k+1} \in f \BH e$,
which is universally reversible since linearly isometric to
$\mathcal{B}(eH)$ (via
Proposition~\ref{TROuOfJCalg} and
\cite[7.4.15]{HOS} or \cite[Corollary 4.5]{HancheOlsenJC}). Hence,
considering the
representation of \JCstar-triples
$f \BH e \to \utr(f \BH e) \subseteq \utr(E)$
and taking $y_i = \utr(x_i)$ for $i = 1, \ldots, 2k+1$,
we have $[y_1 \ldots y_{2k+1}] \in \utr(f \BH e) \subseteq
\utr(E)$. This shows that $E$ is universally reversible, and the
result follows from Corollary~\ref{TROuUnivRevT}.
\qedhere
\end{enumerate}
\end{proof}

\subsection*{Hermitian and symplectic factors}

\begin{theorem}
\label{HermAndSympThm}
Let $H$ be a Hilbert space of dimension $n$, possibly infinite. Let
$E$ denote
\begin{enumerate}[(a)]
\item $S_n = \{x \in \BH : x^t = x\}$ where $2 \leq n \leq \infty$, or
\item $A_n = \{x \in \BH : x^t = -x\}$ where $5 \leq n \leq \infty$.
\end{enumerate}
Then $E$ is universally reversible and $(\TROu(E), \utro E)  = (\BH,
\mathrm{inclusion})$.
\end{theorem}

\begin{proof}
\begin{enumerate}[(a)]
\item 
If $E$ is the Hermitian factor the result is immediate from 
Proposition~\ref{TROuOfJCalg} and the well known facts  \cite[Theorem
2.2 and p. 1070]{HancheOlsenJC} that $E$ is universally reversible and
that $C^*_{\mathrm{J}}(E) = \BH$ with universal \JCstar-algebra embedding $E
\hookrightarrow \BH$.

\item Let $E$ be the symplectic factor $A_n$.
(We could rely on \cite[Lemmas 4.1, 4.2 \&
4.3]{NealRussoTAMS03} for finite $n \geq 5$,
but instead we give a different argument.)

If $n$ is even or infinite, then (see \cite[pp 167--169]{HOS}) there
is a conjugation $\mathfrak{j} \colon H \to H$ with $\mathfrak{j}^2 = -1$ inducing a
quaternionic structure $H_{\IH}$ on $H$ and a $*$-antiautomorphism
$\omega \colon \BH \to \BH$ given by $\omega(x) = - \mathfrak{j} x^*
\mathfrak{j}$ such that
the \JCstar-algebra, $A$, of fixed points of $\omega$ is Jordan $*$
isomorphic to the complexification of the $JC$-algebra
$\mathcal{B}(H_\IH)_{sa}$. By \cite[Theorem
2.2 and p. 1070]{HancheOlsenJC} we have that $A$ is universally
reversible and that $C^*_{\mathrm{J}}(A) = \BH$ with universal embedding $A
\hookrightarrow \BH$.
Moreover $u^* E = A$ where $u$ is the unitary defined by $u(\xi) = -
\overline{ \mathfrak{j}\left( \bar \xi \right)}$ for $\xi \in H$.
Application of Proposition~\ref{TROuOfJCalg} alongside the fact that
$x \mapsto u^* x$ is a TRO automorphism of $\BH$ completes the proof
for $n \geq 6$  even or infinite.

For $n \geq 7$ odd, a straightforward argument applies. Let
$1 \leq r \leq n$, $E = A_n$,
$T = M_n(\IC)$,
$T_{\hat{r}} = \{ x \in T : x_{r,j} = 0 =
x_{j,r} \mbox{ for } 1 \leq j \leq n\}$ (a $*$-subalgebra of $T$
isomorphic to $M_{n-1}(\IC)$) and $E_{\hat{r}} = E \cap T_{\hat{r}}$,
a subtriple of $E$ triple isomorphic to $A_{n-1}$.
Write $E_{i,j}$ ($i, j = 1, \ldots , n$) for the standard
matrix units
of $M_n(\IC)$, $u_{i, j} = \utro E (E_{i,j} - E_{j, i})$.
By the preceeding paragraph we know that
$\TROu(E_{\hat{r}}) = T_{\hat{r}}$ and so there is a TRO homomorphism
$\pi_r \colon T_{\hat{r}} \to \TROu(E)$ with $\pi_r (E_{i,j} - E_{j,
i}) = u_{i,j}$ for $i$ and $j$ both different from $r$, and if $k
\notin \{i, j , r\}$
we have
\[
\trop {u_{i,k}} {u_{i,k}} {u_{i, j}} = \pi_r(E_{i, j}),
\]
hence that $f_{i, j} = \pi_r(E_{i, j})$ is independent of $r \notin
\{i, j\}$ (for $i \neq j$). From $\trop{f_{i,j}}{f_{k, j}}{f_{k,i}} =
\pi_r(E_{i,i})$ (for distinct $i$, $j$, $k$ and $r$) we see that
$f_{i, i} = \pi_r(E_{i,i})$ is also independent of $r \neq i$.
It follows easily that $f_{i, j}$ obey the same TRO
relations as $E_{i, j}$ and hence that the map $M_n(\IC) \to
\TROu(E)$ ($E_{i, j} \mapsto f_{i, j}$) is a TRO isomorphism.
As $x \mapsto - x^t$ is an antiautomorphism of $M_n(\IC)$ with fixed
point set $A_n$, we have that $A_n$ is universally reversible.

For the remaining case, $E = A_5$, a more complicated argument is
required as $A_4$ is triple isomorphic to the
spin factor $V_5$. Working in $A_4 \subset
M_4$, write $v_{i, j} = E_{i, j} - E_{j, i}$ ($1 \leq i \neq j \leq
4$) and $v = v_{1,3} + v_{2,4}$. Then $v$ is unitary,
$s'_1 =  v^*(v_{1,3} - v_{2,4})$,
$s'_2 = v^*(v_{2,3} + v_{1,4})$,
$s'_3 = i v^*(v_{2,3} - v_{1,4})$,
$s'_4 = v^*(v_{1,2} + v_{3,4})$,
$s'_5 = i v^*(v_{1,2} - v_{3,4})$,
are 5 anticommuting symmetries
in $v^* A_4$ which together with $s'_0 = v^*v = 1_4$, span $v^*A_4$.
This provides a Jordan $*$ isomorphism $v^* A_4 \to V_5$ (notation as
in Lemma~\ref{OddSpinsLemma}), mapping $s_i'
\mapsto s_i$ ($1 \leq i \leq 4$) and $s'_5 \mapsto t_5$. Under the
triple isomorphism $A_4 \to v^* A_4 \to V_5$, the map $\beta_2$
corresponds to the isometry of $A_4$ exchanging $v_{1,2}$ and
$v_{3,4}$ but leaving the other $v_{i,j}$ (with $i < j$) fixed. From
Lemma~\ref{OddSpinsLemma} we may realise $\utro{A_4} \colon A_4 \to
M_4(\IC) \oplus M_4(\IC)$ as the linear map given by
$v_{i,j} \mapsto v_{i,j} \oplus
v_{i,j}$ when $i < j$ satisfy $(i, j) \notin \{ (1,2), (3,4) \}$, $v_{1,2}
\mapsto v_{1,2} \oplus v_{3,4}$, $v_{3,4} \mapsto v_{3,4} \oplus
v_{1,2}$.

Let $h_{i,j}$ denote the matrix units $E_{i,j} \oplus 0 \in
\TROu(A_4) = M_4(\IC) \oplus M_4(\IC)$ and $g_{i, j} = 0 \oplus
E_{i,j} \in \TROu(A_4)$ ($1 \leq i, j \leq 4$).
It follows that for $w_{i,j} = \utro{A_4}(v_{i,j})$ we have
\[
w_{i,j} = (h_{i, j} - h_{j, i}) \oplus (g_{i, j} - g_{j, i})
\quad (i < j, (i, j) \notin \{ (1,2), (3,4) \}),
\]
$w_{1,2} = (h_{1, 2} - h_{1, 2}) \oplus (g_{3, 4} - g_{4, 3})$,
$w_{3,4} = (h_{3, 4} - h_{3, 4}) \oplus (g_{1, 2} - g_{2, 1})$,
$
w_{3,4} w_{3,4}^* + 
w_{2,3} w_{2,3}^* + 
w_{2,4} w_{2,4}^* = 2(h_{2,2} + h_{3,3} + h_{4,4}) \oplus (g_{1,1} +
3 g_{2,2} +  g_{3,3} + g_{4,4})$, and hence that there is a polynomial
$P(z) $ without constant term so that $(1_4 - h_{1,1})
\oplus 1_4
= P( w_{3,4} w_{3,4}^* +
w_{2,3} w_{2,3}^* +
w_{2,4} w_{2,4}^*)
$.
Similarly $(1_4 - h_{1,1})
\oplus 1_4 = P( w_{3,4}^* w_{3,4} + w_{2,3}^* w_{2,3} + w_{2,4}^*
w_{2,4})$.

Now we turn to the triple embedding
$\pi_5 \colon A_4 \to \TROu(A_5)$ ($x \mapsto \utro{A_5} (x \oplus
0_1)$) (where $0_1$ means the $1 \times 1$ zero matrix)
and the induced TRO homomorphism
$\tilde{\pi}_5 \colon \TROu(A_4) \to \TROu(A_5)$.
For $1 \leq i \neq j \leq 5$, write $u_{i,j} = E_{i,j} - E_{j,i} \in
A_5$ and $W_{i,j} = \utro{A_5} (u_{i,j}) \in \TROu(A_5)$. Note
$\tilde{\pi}_5(w_{i,j}) = \tilde{\pi}_5( \utro{A_4}(v_{i,j}) ) =
\pi_5( v_{i,j})  = \utro{A_5}(u_{i,j}) = W_{i,j} $ (for $1 \leq i < j \leq
4$).
Since the TRO
homomorphism $\tilde{\pi}_5$ induces $*$-homomorphisms on the left and
right $C^*$-algebras,
\[
p
= P( W_{3,4} W_{3,4}^* + W_{2,3} W_{2,3}^* + W_{2,4} W_{2,4}^*)
\]
is a projection in the left $C^*$-algebra of $\TROu(A_5)$ and
\[
q
= P( W_{3,4}^* W_{3,4} + W_{2,3}^* W_{2,3} + W_{2,4}^* W_{2,4})
\]
is a projection in the right $C^*$-algebra of $\TROu(A_5)$. As
$u_{1,5}$ is orthogonal in $A_5$ to $u_{3,4}$, $u_{2,3}$ and
$u_{2,4}$, we have that $W_{1,5}$ is orthogonal in $\TROu(A_5)$ to $W_{3,4}$,
$W_{2,3}$ and $W_{2,4}$, hence that $p W_{1,5} = 0$. Similarly
$W_{1,5} q = 0$.
Let $H_{i,j} = \tilde{\pi}_5(h_{i,j})$, $G_{i,j} =
\tilde{\pi}_5(g_{i,j})$ ($1 \leq i, j \leq 4$).
As $G_{i,j} G_{i,j}^* W_{1,5} =
G_{i,j} G_{i,j}^* p W_{1,5} =0$
and
$W_{1,5}  G_{i,j}^* G_{i,j} = W_{1,5} q G_{i,j}^* G_{i,j}  = 0$,
we have that $W_{1,5}$ is
orthogonal to all $G_{i,j}$ ($1 \leq i, j \leq 4$).
As
$ 2 \tp { u_{1,5}} { u_{1,5}} {u_{1,2}} = u_{1,2}$,
we can then apply $\utro {A_5}$ to
conclude that
$2 \tp { W_{1,5}} { W_{1,5}} {H_{1,2} - H_{2,1}} = H_{1,2} -
H_{2,1} + G_{1,2} - G_{2,1}$, then multiply on the left
by $(G_{1,1}+ G_{2,2})(G_{1,1}+ G_{2,2})^*$ and on
the right
by $(G_{1,1}+ G_{2,2})^*(G_{1,1}+ G_{2,2})$
to get
$0= G_{1,2} - G_{2,1}$,
from which it follows that $G_{i, j} = 0$ for all $1 \leq i,
j \leq 4$.

Thus the subTRO of $\TROu(A_5)$ generated by $\{ W_{i,j} : 1 \leq i, j
\leq 4\}$ is a copy of $M_4(\IC)$ and we can recover the elements
$H_{i,j}$ (which satisfy the same TRO rules as $E_{i, j}$  ($1 \leq i,
j \leq 4$)) via $H_{i,j} = \trop{W_{i,k}}{W_{i,k}}{W_{i, j}}$,
$H_{i, i}  = \trop{H_{i,j}}{H_{i,j}}{H_{i, k}}$
(for distinct $i, j, k$ ($1 \leq i, j, k \leq 4$)).
Similar matrix units $H^r_{i, j}$
arise from all four element subsets $K_r = \{1, 2, 3,
4, 5\} \setminus \{r\}$ (now with $(i, j) \in K_r \times K_r$)
and it is quite straightforward to check that $H^r_{i, j}$ is
independent of the choice of $r \in \{1, 2, 3,
4, 5\} \setminus \{i, j\}$, and then that the TRO generated by
$\utro{A_5}(A_5)$ must be a copy of $M_5(\IC)$.
Finally note that that $A_5$ is reversible in $M_5(\IC)$ to complete
the proof.
\qedhere
\end{enumerate}
\end{proof}

\subsection*{Reversibility}

Collecting information from Theorems
\ref{TROuHTheorem}, \ref{TROuRmn} and \ref{HermAndSympThm},
and Lemma~\ref{SpinsRevLemma},
we have the following classification of the Cartan
factors which are universally reverible \JCstar-triples.

\begin{theorem}
\label{CharactRevCFs}
Spin factors of dimension greater than 4 and Hilbert spaces of
dimension greater than 2 are not universally reversible. All other
Cartan factors are universally reversible.
\end{theorem}

\section{Operator space structures}
\label{secOpSpStructs}

We consider operator space structures of \JCstar-triples arising from
concrete triple embeddings in $C^*$-algebras, making essential use of
universal TROs and establishing links with injective envelopes and
triple envelopes, concentrating upon Cartan factors in the latter part
of the section. Frequent use is made of the coincidence of complete
isometries between TROs and TRO isomorphisms.

\begin{definition}
\label{DefOfJCStarOpSp}
A \emph{\JC -operator space structure} on a \JCstar-triple $E$ is
an operator space structure determined by a linear isometry from $E$
onto a \JCstar-subtriple of $\BH$.
By a \emph{\JC-operator space} is meant a \JCstar -triple $E$
together with a
prescribed \JC -operator space structure on $E$.
\end{definition}

Let $E$ be a \JCstar-triple. Two \JC -operator space structures
$E_1$ and $E_2$ are deemed equal, written $E_1 = E_2$, if and only if
the identity map $E_1 \to E_2$ is a complete isometry.
We shall make extensive use of the (norm closed) ideals $\mathcal{I}$
of $\TROu(E)$ for which $\utro E(E) \cap \mathcal{I} = \{0\}$ which,
henceforth, we shall refer to as the \emph{operator space ideals} of
$\TROu(E)$.
For each operator space ideal $\mathcal{I}$
of $\TROu(E)$ we have the \JC-operator space
structure,
$E_{\mathcal{I}}$,
on $E$ determined by the isometric embedding $E \to
\TROu(E)/\mathcal{I}$ ($x \mapsto \utro E (x) + \mathcal{I}$), the
\JCstar-triple image of which we denote $\tilde{E}_{\mathcal{I}}$. We
note that $\TRO(\tilde{E}_{\mathcal{I}}) = \TROu(E)/\mathcal{I}$.
Given operator space ideals $\mathcal{I}$ and $\mathcal{J}$ of
$\TROu(E)$, these notations imply that $E_{\mathcal{I}} =
E_{\mathcal{J}}$ if and only if $\tilde{E}_{\mathcal{I}}  \to
\tilde{E}_{\mathcal{J}}$ ($\utro E (x) + \mathcal{I} \mapsto \utro E
(x) + \mathcal{J}$) is a complete isometry.

By the functorial properties of the universal TRO, if $\pi \colon E
\to F$ is a linear isometry onto a \JCstar-triple $F$, then $\TROu(\pi)
\colon \TROu(E) \to \TROu(F)$ preserves operator space ideals.
By Proposition~\ref{MaximalOpSpIdeals} we note that every operator
space ideal in $\TROu(E)$ is contained in a  maximal operator space
ideal.

\begin{proposition}
\label{PropIdealsGiveOpSpaceStructs}
Let $E$ be a \JCstar -triple and let
$\pi \colon E \to F$
be a linear isometry onto a \JCstar -subtriple $F$ of $\BH$ (regarded
as an operator subspace of $\BH$). Then there is an operator space
ideal
$\mathcal{I} $ of $ \TROu(E)$
such that $\pi \colon E_{\mathcal{I}} \to F$ is a
complete isometry. Hence every \JC-operator space structure on $E$
equals ${E}_{\mathcal{I}}$ for some operator space ideal
$\mathcal{I}$.
\end{proposition}

\begin{proof}
Letting $\mathcal{I} = \ker \tilde{\pi}$ where $\tilde{\pi} \colon
\TROu(E) \to \TRO(F)$ is the (surjective) TRO homomorphism with
$\tilde{\pi} \circ \utro E = \pi$, we have that $\mathcal{I}$ is an
operator space ideal of $\TROu(E)$ and the induced TRO isomorphism
$\TROu(E)/\mathcal{I} \to \TRO(F)$ restricts (via
Lemma~\ref{TROmorphsAreCC}) to the complete isometry
$\tilde{E}_{\mathcal{I}} \to F$ ($\utro E (x) + \mathcal{I} \mapsto
\pi(x)$).
\end{proof}

\begin{proposition}
\label{CompIsomEIs}
Let $\pi \colon E \to F$ be a linear isometry between \JCstar-triples.
Let $\mathcal{I}$ and $\mathcal{J}$ be operator space ideals of
$\TROu(E)$ and $\TROu(F)$, respectively, and let $\mathcal{K} =
\TROu(\pi)(\mathcal{I})$. Then $\pi \colon E_{\mathcal{I}} \to
F_{\mathcal{J}}$ is a complete isometry if and only if
$F_{\mathcal{J}} = F_{\mathcal{K}}$.
\end{proposition}

\begin{proof}
The natural TRO homomorphism $\TROu(E)/\mathcal{I} \to
\TROu(F)/\mathcal{K}$ induced by $\TROu(\pi)$ restricts to the
complete isometry $\tilde{E}_{\mathcal{I}} \to
\tilde{F}_{\mathcal{K}}$ ($\utro E (x) + \mathcal {I} \mapsto \utro F
(\pi(x)) + \mathcal{K}$, implying that $\pi \colon  E_{\mathcal{I}}
\to F_{\mathcal{K}}$ is a complete isometry, whence the result.
\end{proof}

\begin{proposition}
\label{compareEIs}
Let $E$ be a \JCstar-triple and let
$\mathcal{I}$ and $\mathcal{J}$ be operator space ideals of
$\TROu(E)$  with $\mathcal{I} \subseteq \mathcal{J}$.
Then the identity map $\pi \colon E_{\mathcal{I}} \to
E_{\mathcal{J}}$ is a complete contraction.
\end{proposition}

\begin{proof}
$\tilde{E}_{\mathcal{I}} \to \tilde{E}_{\mathcal{J}}$ ($\utro E +
\mathcal {I} \mapsto \utro E (x) + \mathcal{J}$) is the restriction of
the quotient TRO homomorphism $\TROu(E)/\mathcal{I} \to
\TROu(E)/\mathcal{J}$ and so is a complete contraction.
\end{proof}

Given a \JCstar -triple $E$ we denote by $\mbox{MAX}_{JC}(E)$ the
\JC-operator space structure determined by the zero ideal in
$\TROu(E)$. By Proposition~\ref{compareEIs} the identity map
$\mbox{MAX}_{JC}(E) \to E$ is a complete contraction for every
\JC-operator space structure on $E$. By Proposition~\ref{CompIsomEIs}
every surjective linear isometry $\pi \colon \mbox{MAX}_{JC}(E) \to
\mbox{MAX}_{JC}(E)$ is a complete isometry.

Recall that the injective envelope $(I(V), j)$ of an operator space $V$
is an injective operator space $I(V)$ together with a completely
isometric injection, $j \colon V \to I(V)$, such that $I(V)$ is
minimal injective containing $j(V)$.
Moreover $(I(V), j)$ is unique up to complete isometry and $I(V)$ may
be realised as a TRO; in which case $\TRO(j(V))$ is said to be the
\emph{triple envelope},  $\mathcal{T}(V)$, associated with $V$ and it
possesses the following universal property \cite{Hamana99} (see also
\cite[Theorem 8.3.9]{BlecherleMerdy})
\begin{quote}
Given an operator space $U \subset \BH$ and a complete isometry $\phi
\colon U \to V$ onto $V$, there is a unique TRO homomorphism
$\hat{\phi} \colon \TRO(U) \to \mathcal{T}(V)$ such that $j \circ \phi
= \hat{\phi} |_U$.
\end{quote}

\begin{theorem}
\label{TripleEnvThm}
Let $(I(E), j)$ be an injective envelope of a \JC-operator space $E$,
where $I(E)$ is a TRO. Then
\begin{enumerate}[(a)]
\item $j \colon E \to I(E)$ is a triple homomorphism;
\item there is a unique TRO homomorphism $\tilde{j} \colon \TROu(E)
\to \mathcal{T}(E)$ such that $\tilde{j} \circ \utro E = j$, and
$\tilde{j}$ is surjective;
\item $\ker \tilde{j}$ is the largest operator space ideal
$\mathcal{I}$ of
$\TROu(E)$ with $E = E_{\mathcal{I}}$.
\end{enumerate}
\end{theorem}

\begin{proof}
\begin{enumerate}[(a)]
\item
We may suppose $E$ is a \JCstar-subtriple (and operator subspace) of
$\BH$ and may choose an injective envelope $W$ of $E$ in $\BH$
together with a completely contractive projection $P$ on $\BH$ with
image $W$. There is a complete isometry $\phi \colon W \to I(E)$ such
that $\phi |_E = j$ and $\phi( P(\trop a b c) ) = \trop { \phi(a)} {
\phi(b)} {\phi(c)}$ for all $a, b , c \in W$ \cite{Youngson83}. In
particular, for each $x$ in $E$ we have $\phi(\tp x x x) =
\phi(P (\trop x x x) ) = \trop {\phi(x)} {\phi(x)} {\phi(x)}$, proving
(a).

\item
This is immediate from (a) and the universal property of $\TROu(E)$.

\item
Letting $\mathcal{I} = \ker \tilde{j}$, the map
$\tilde{E}_{\mathcal{I}} \to
j(E)$ ($\utro E (x) + \mathcal{I} \mapsto j(x)$) is a complete
isometry as in the proof of
Proposition~\ref{PropIdealsGiveOpSpaceStructs} so that composing with
$j^{-1} \colon j(E) \to E$ we have that the identity map
$E_{\mathcal{I}} \to E$ is a complete isometry.

Let $\mathcal{J}$ be any operator space ideal of $\TROu(E)$ such that
$E_{\mathcal{J}} = E$, and let $q \colon \TROu(E) \to
\TROu(E)/\mathcal{J}$ be the quotient map. Since $\pi \colon
\tilde{E}_{\mathcal{J}} \to E$ ($\utro E (x) + \mathcal{J} \mapsto x$)
is a complete isometry, the above-mentioned universal property of
$\mathcal{T}(E)$ implies that there is a TRO homomorphism $\hat{\pi}
\colon \TRO( \tilde{E}_{\mathcal{J}} ) ( = \TROu(E)/\mathcal{J}) \to
\mathcal{T}(E)$ such that $\pi \circ q \circ \utro E = j$. Hence
$\hat{\pi} \circ q = \tilde{j}$ by (b), proving that $\mathcal{J}
\subseteq \ker \tilde{j}$.
\qedhere
\end{enumerate}
\end{proof}

We remark that in the notation of Theorem~\ref{TripleEnvThm}, $\ker
(\tilde{j})/\mathcal{J}$ corresponds to the \v{S}ilow boundary
\cite{Hamana99} of $E_{\mathcal{J}}$.

A natural question arises from our prior discussion. Do there exist
\JCstar-triples $E$ for which there are operator space ideals
$\mathcal{I} \subsetneq \mathcal{J} \subset \TROu(E)$  with
$E_{\mathcal{I}} = E_{\mathcal{J}}$? (Equivalently, do there exist
\JC-operator spaces with non-zero \v{S}ilow boundary?)

As an application we shall show that there are no such examples among
non-Hilbertian Cartan factors and shall completely describe the 
\JC-operator spaces in these cases. For the finite rank non-Hilbertian
Cartan
factors our results may be seen as complements of those of
\cite{NealRussoTAMS03} and
\cite[\S7]{LeMerdyRicardRoydorv2}.

\begin{proposition}
\label{MAXJCisUnique}
Let $E$ be a Cartan factor that is not rectangular of dimension
greater than one nor an even dimensional spin factor. Then 
$\mbox{MAX}_{JC}(E)$ is the unique \JC-operator space structure on
$E$.
\end{proposition}

\begin{proof}
By the results of \S\ref{SectionCartanFs}, $\TROu(E)$ has no non-zero
operator space ideals. In the infinite dimensional hermitian and
symplectic cases this is because for every non-zero ideal
$\mathcal{J}$ of $\TROu(E) = \BH$, $\mathcal{K}(H) \subseteq
\mathcal{J}$ and so $\utro E(E) \cap \mathcal{J}$ contains
non-trivial compact operators.
\end{proof}

By \cite[Proposition 7.1]{LeMerdyRicardRoydorv2}
(with different notation) the isometry
$\beta_n$ of Lemma~\ref{OddSpinsLemma}
is not a complete isometry, an alternative
proof of
which is included in the next result which is a complement of
\cite[Proposition 7.3 (1),(2)]{LeMerdyRicardRoydorv2}.

\begin{proposition}
\label{SpinOpSpace}
For $n = 1, 2, \ldots$,
the spin factor $V_{2n+1}$  has precisely three \JC-operator space
structures, two of which are completely isometric.
\end{proposition}

\begin{proof}
It is convenient to recall the Jordan $*$ isomorphic copy,
$\tilde{V}_{2n+1}$, of $V_{2n+1}$ and to recall the linear isometry
$\beta_n \colon \tilde{V}_{2n+1} \to \tilde{V}_{2n+1}$ of
Lemma~\ref{OddSpinsLemma} which is not the identity map but does act
identically on $V_{2n}$. Since $V_{2n} \subset \tilde{V}_{2n+1}
\subset M_{2^n}(\IC) = \TROu(V_{2n})$, it follows from
Theorem~\ref{TripleEnvThm} that $M_{2^n}(\IC)$ is a triple envelope of
$V_{2n}$. If $\beta_n$ is a complete contraction, it has a completely
contractive extension $\phi \colon M_{2^n}(\IC) \to M_{2^n}(\IC)$
\cite[Theorem 4.1.5]{ER} which, since it acts
identically on $V_{2n}$ must be the identity map \cite[Theorem
6.2.1]{ER}, as therefore must $\beta_n$, a
contradiction. Hence, $\beta_n$ is not completely contractive.

Since the projection maps $M_{2^n}(\IC) \oplus M_{2^n}(\IC) \to
M_{2^n}(\IC)$ ($x \oplus y \mapsto x$, $x \oplus y \mapsto y$) are
complete contractions, it follows that $\mu_n, \pi \colon
\tilde{V}_{2n+1} \to M_{2^n}(\IC) \oplus M_{2^n}(\IC)$ are not
completely contractive, where $\mu_n(x) = x \oplus \beta_n(x)$ (the
canonical embedding) and $\pi(\beta_n(x)) = \mu_n(x)$, for all $x \in
\tilde{V}_{2n+1}$. Therefore the inclusion $\tilde{V}_{2n+1} \subset
M_{2^n}(\IC)$, $\beta_n$ and $\mu_n$ determine three distinct operator
space structures. In the same order, these are the structures
determined by the complete set of operator space ideals
$\{0\} \oplus M_{2^n}(\IC)$, $M_{2^n}(\IC) \oplus \{0\}$ and the zero
ideal.
The first two mentioned \JC-operator space structures are completely
isometric by Proposition~\ref{CompIsomEIs} since $\TROu(\beta_n)$ is
the exchange automorphism of $\TROu(\tilde{V}_{2n+1}) = M_{2^n}(\IC)
\oplus M_{2^n}(\IC)$.
\end{proof}

Projections $e$ and $f$ in $\BH$ are equivalent if and only if $eH$
and $fH$ have the same dimension. Since an ideal $\mathcal{I}$ of
$\BH$ is the norm closed linear span of its projections
and since a projection $e \in \mathcal{I}$ implies $e^t \in
\mathcal{I}$, because $e^t \sim e$, we have $\mathcal{I}^t =
\mathcal{I}$. Given ideals $\mathcal{I}$ and $\mathcal{J}$ of $\BH$ we
have $\mathcal{I} \subseteq \mathcal{J}$ or
$\mathcal{J} \subseteq \mathcal{I}$ with $\mathcal{K}(H)$ being the
minimal non-zero ideal. If $\mathcal{I} \subsetneq \mathcal{J}$ we may
choose equivalent projections $e$, $f$ in $\mathcal{J}$ with $e, f
\notin \mathcal{I}$ and hence a $*$ isomorphic copy, $M$, of
$M_2(\IC)$ with $M \subset \mathcal{J}$ and $M \cap \mathcal{I} =
\{0\}$. (If $u \in \BH$ with $e = u^* u$, $f = u u^*$, then $u \in
\mathcal{J}$ and the linear span of
$\{e, f, u, u^*\}$ is $*$-isometric to
$M_2(\IC)$.) We further note that if $U$ is an operator subspace of
$\BH$ containing a completely isometric copy
of $M_2(\IC)$ and if
(temporarily) $U_t$ and $U_\delta$ denote the operator space
structures induced on $U$ by $U \to \BH$ ($x \mapsto x^t$) and $U \to
\BH \oplus \BH$
($x \mapsto x \oplus x^t$), respectively, then the 
the identity maps $U \to U_t$, $U \to U_\delta$ and $U_t \to U_\delta$
are not complete isometries.

\begin{theorem}
\label{ThmClassifOPSTR}
Let $E$ be a Cartan factor such that $E$ is not a Hilbert space.
Then $\mathcal{I}
\leftrightarrow E_{\mathcal{I}}$ is a bijective correspondence
between
the operator space ideals of $\TROu(E)$ and the \JC-operator space
structures of $E$.
\end{theorem}

\begin{proof}
The other cases being settled by Proposition~\ref{MAXJCisUnique} and
\ref{SpinOpSpace}, we may suppose that $E = \BH e$ where $e$ is a
projection in $\BH$ of rank not less than two. By
Theorem~\ref{TROuRmn}, $\TROu(E) = E \oplus E^t$ with $\utro E (x) = x
\oplus x^t$, and we may assume that $e^t = e$. The canonical
involution of $\TROu(E)$ is given by $\Phi(x \oplus y^t) = y \oplus
x^t)$. By Proposition~\ref{PropTROIdeals} the ideals of $\BH e$ are
the $\mathcal{I} e$ where $\mathcal{I}$ ranges over ideals of $\BH$,
forming a chain with $\mathcal{K}(H) e$ being the minimal non-zero
ideal. The operator space ideals of $\TROu(E)$ are $\mathcal{A} \oplus
\{0\}$ and $\{0\} \oplus \mathcal{A}$ where $\mathcal{A}$ is an ideal
of $E$.

Let $\mathcal{I}$ and $\mathcal{J}$ be ideals of $\BH$ such that
$\mathcal{I} e \subsetneq \mathcal{J} e$. The $e\mathcal{I} e
\subsetneq e\mathcal{J} e$ (see Proposition~\ref{PropTROIdeals}).
Passing to $e \BH e$ ($\equiv \mathcal{B}(eH)$) remarks preceding the
statement show that we may choose a $*$-isomorphic copy, $M$, of
$M_2(\IC)$ in $e\mathcal{J} e$ such that $M \cap e \mathcal{I} e =
\{0\}$, and hence that $M \cap \mathcal{I} e =
\{0\}$. The quotient map $\BH e \to \BH e /\mathcal{I} e$ restricts to
a $*$-isomorphism on $M$.

Put $\mathcal{R} = \mathcal{I} e \oplus \{0\}$ and
$\mathcal{S} = \mathcal{J} e \oplus \{0\}$.
We have $\Phi(\mathcal{R}) = \{0\} \oplus e \mathcal{I}$ and
$\Phi(\mathcal{S}) = \{0\} \oplus e \mathcal{J}$.
Suppose $\mathcal{I} e \neq \{0\}$. The maps inducing the arising
\JC-operator space structures are as indicated below:
\begin{eqnarray*}
E_{\mathcal{R}}:  x \mapsto (x + \mathcal{I} e) \oplus x^t &;&
E_{\Phi(\mathcal{R})}:  x \mapsto x \oplus (x^t + e\mathcal{I} )\\
E_{\mathcal{S}}:  x \mapsto (x + \mathcal{J} e) \oplus x^t &;&
E_{\Phi(\mathcal{S})}:  x \mapsto x \oplus (x^t + e\mathcal{J} )
\end{eqnarray*}
On $e \mathcal{K}(H) e$, $E_{\mathcal{R}}$ and $E_{\Phi(\mathcal{R})}$
induce the structures determined by $x \mapsto x^t$ and $x \mapsto x$,
respectively. Thus $E_{\mathcal{R}} \neq E_{\Phi(\mathcal{R})}$ ($e
\mathcal{K}(H) e$ contains a $*$-isomorphic copy of $M_2(\IC)$).
Similarly $E_{\mathcal{R}} \neq E_{\Phi(\mathcal{S})}$,
$E_{\mathcal{S}} \neq E_{\Phi(\mathcal{R})}$  and $E_{\mathcal{S}}
\neq E_{\Phi(\mathcal{S})}$.

On $M$, both $E_{\mathcal{R}}$ and $E_{\Phi(\mathcal{R})}$ induce the
structure determined by $x \mapsto x \oplus x^t$ whilst 
$E_{\mathcal{S}}$ and $E_{\Phi(\mathcal{S})}$, respectively, induce
the structures determined by $x \mapsto x^t$ and $x \mapsto x$. It
follows that $E_{\mathcal{R}}$, $E_{\Phi(\mathcal{R})}$,
$E_{\mathcal{S}}$ and $E_{\Phi(\mathcal{S})}$ are distinct.

If $\mathcal{I} e = \{0\}$, then $E_{\mathcal{R}}$ ($=
E_{\Phi(\mathcal{R})}$) is determined by $x \mapsto x \oplus x^t$
which differs from $E_{\mathcal{S}}$ and $E_{\Phi(\mathcal{S})}$ on
$e \mathcal{K}(H) e$. This completes the proof.
\end{proof}

\begin{remark}
Consider a \JCstar-triple $E$ linearly isometric to $\BH e$, where $e$ is a
non-zero projection in $\BH$ not of rank one. By
Theorem~\ref{ThmClassifOPSTR} and its
proof, since the ideals of $E$ are in bijective correspondence with
those of $e \BH e$ ($\equiv \mathcal{B}(eH)$),
we may deduce the following, which answers
a question raised in
\cite[Remark 7.4]{LeMerdyRicardRoydorv2}.

\begin{enumerate}[(i)]
\item \emph{If $e$ has finite rank then $E$ has a total of three
distinct operator space strucures.}

\item \emph{If $eH$ is separably infinite then $E$ has a total of five
distinct operator space structures.}

\item \emph{If $eH$ is non-separable, and $\dim(eH) = \aleph_\alpha$,
 the cardinality of
the distinct operator space structures on $E$ is the cardinality
of the ordinal segment $[0,\alpha]$.}
\end{enumerate}

Together with Propositions~\ref{MAXJCisUnique} and
\ref{SpinOpSpace} this provides a complete description of the
\JC-operator spaces of all non-Hilbertian Cartan factors.
We shall extend
Theorem~\ref{ThmClassifOPSTR} to
Hilbertian Cartan factors in a forthcoming paper.
\end{remark}

\begin{corollary}
Let $E$ be a \JCstar-subtriple of a $C^*$-algebra, where $E$ is a
Cartan factor of type $R_{m,n}$ with $m$ countable and $1 < m < n$. Then
every triple automorphism of $E$ is a complete
isometry.
\end{corollary}

\begin{proof}
Let $\pi \colon E \to E$ and $\phi \colon F \to E$ be surjective
linear isometries, where $F = \BH e$ for a projection $e$ in $\BH$
such that $eH$ is separable with $1 < \dim(eH) < \dim (H)$. Let
$\theta, \psi \colon \TROu(F) \to \TRO(E)$ be the TRO homomorphisms
such that $\theta \circ \utro F = \phi$
and $\psi \circ \utro F = \pi
\circ \phi$ and let $\mathcal{I} = \ker \theta$, $\mathcal{J} = \ker
\psi$. We claim that $\mathcal{I} = \mathcal{J}$. The argument being
simpler when $eH$ has finite dimension, suppose that $eH$ is separably
infinite. The five operator space ideals of $\TROu(F)$ are the zero
ideal, $F \oplus \{0\}$, $\{0\} \oplus F^t$, $\mathcal{K} \oplus
\{0\}$ and $\{ 0\} \oplus \mathcal{K}^t$, where $\mathcal{K}$ denotes
$\mathcal{K}(H) e$. The corresponding quotient TROs are $F \oplus
F^t$, $F^t$, $F$, $F/\mathcal{K} \oplus F^t$ and $F \oplus
F^t/\mathcal{K}^t$ which are mutually non-isomorphic as TROs since $F$
and $F^t$ are not TRO isomorphic (as $m \neq n$)
and $F/\mathcal{K}$,
$F^t/\mathcal{K}^t$ do not possess minimal tripotents.
Thus, as $\mathcal{I}$ and $\mathcal{J}$ are  operator space ideals
with $\TROu(F)/\mathcal{I}$ TRO isomorphic to $\TROu(F)/\mathcal{J}$,
we have $\mathcal{I} = \mathcal{J}$. Hence if $\theta_{\mathcal{I}},
\psi_{\mathcal{I}} \colon \TROu(F)/\mathcal{I} \to \TRO(E)$ denote the
induced TRO isomorphisms, $\psi_{\mathcal{I}} \circ
\theta_{\mathcal{I}}^{-1}$
is a TRO automorphism of $\TRO(E)$ extending
$\pi$.
\end{proof}

\begin{theorem}
Let $E$ be a \JCstar-subtriple of a $C^*$-algebra where $E$ is a
Cartan factor which is not a Hilbert space.
We have the following.
\begin{enumerate}[(a)]
\item If $(I(E), j)$ is an injective envelope of $E$ where $I(E)$ is a
TRO, then $j \colon E \to I(E)$ extends to a TRO isomorphism $\psi
\colon \TRO(E) \to \mathcal{T}(E)$.
\item $E = \mbox{MAX}_{JC}(E)$ if and only if $(\TROu(E), \utro E)
\equiv (\TRO(E), \mathrm{inclusion})$.
\item Suppose $E = \mbox{MAX}_{JC}(E)$ or $E$ is reflexive. If $E$ is
not an infinite dimensional spin factor, then $(\TRO(E),
\mathrm{inclusion})$ is an injective envelope of $E$ and $\TRO(E) =
\mathcal{T}(E)$.
\end{enumerate}
\end{theorem}

\begin{proof}
\begin{enumerate}[(a)]
\item In that case, by Proposition~\ref{PropIdealsGiveOpSpaceStructs}
together with Theorems \ref{TripleEnvThm} and \ref{ThmClassifOPSTR},
we have $E = E_{\mathcal{I}}$, where $\mathcal{I} = \ker \tilde{j}$,
and TRO isomorphisms $\pi \colon \TROu(E)/\mathcal{I} \to
\TRO(E)$,
$\theta \colon \TROu(E)/\mathcal{I} \to \mathcal{T}(E)$ such that
$\pi(\utro E (x) + \mathcal{I}) = x$,
$\theta(\utro E (x) + \mathcal{I}) = j(x)$ for all $x \in E$. Thus
$\theta \circ \pi^{-1}$ is a TRO isomorphism extending $j$.

\item
If $E = \mbox{MAX}_{JC}(E)$ then by
Proposition~\ref{PropIdealsGiveOpSpaceStructs} and
Theorem~\ref{ThmClassifOPSTR} the TRO homomorphism $\pi \colon
\TROu(E) \to \TRO(E)$ such that $\pi \circ \utro E = \id_E$ is an
isomorphism. The converse is clear.

\item
Suppose that $E$ is not an infinite dimensional spin factor. Then, by
the results of \S\ref{SectionCartanFs}, $\TROu(E)$ is an injective
TRO, as is $\TRO(E)$ when $E$ is reflexive. Thus (c) is a consequence
of (a) and (b).
\qedhere
\end{enumerate}
\end{proof}

When finalising this work we learned that, by other means, D. Bohle and
W. Werner have independently observed the existence of a universal TRO of
a \JCstar-triple and computed the universal TRO of finite dimensional
\JCstar-triples.

\def\cprime{$'$} \def\polhk#1{\setbox0=\hbox{#1}{\ooalign{\hidewidth
  \lower1.5ex\hbox{`}\hidewidth\crcr\unhbox0}}}
  \def\polhk#1{\setbox0=\hbox{#1}{\ooalign{\hidewidth
  \lower1.5ex\hbox{`}\hidewidth\crcr\unhbox0}}}

\end{document}